\documentclass[a4paper,reqno]{amsart}

\usepackage{amsmath}
\usepackage{amssymb}
\usepackage{amsthm}
\usepackage{txfonts}
\usepackage{microtype}
\usepackage[usenames,dvipsnames]{xcolor}
\usepackage{enumerate}
\usepackage{graphicx}
\usepackage[hidelinks]{hyperref}
	\hypersetup{colorlinks={true},linkcolor=purple,citecolor=purple}
\usepackage{mathrsfs}
\usepackage{dsfont}
\usepackage{cleveref}
\usepackage{fullpage}
\usepackage[foot]{amsaddr}
\usepackage{tikz}
\usepackage[utf8]{inputenc}
\usepackage{enumitem}
\usepackage{caption}
\usepackage[labelformat=empty]{subcaption}

\theoremstyle{definition}
\newtheorem{defin}{Definition}[section]
\newtheorem{rem}[defin]{Remark}
\newtheorem{ex}[defin]{Example}

\theoremstyle{plain} 
\newtheorem{prop}[defin]{Proposition}
\newtheorem{lem}[defin]{Lemma}
\newtheorem{cor}[defin]{Corollary}
\newtheorem{thm}[defin]{Theorem}

\Crefname{defin}{Definition}{Definition}
\Crefname{rem}{Remark}{Remarks}
\Crefname{ex}{Example}{Examples}
\Crefname{prop}{Proposition}{Propositions}
\Crefname{lem}{Lemma}{Lemmata}
\Crefname{cor}{Corollary}{Corollaries}
\Crefname{thm}{Theorem}{Theorems}

\def\N{\mathbb N}
\def\R{\mathbb R}
\def\Z{\mathbb Z}
\def\E{\mathcal E}
\def\I{\mathcal I}

\setlist[enumerate,1]{label={(\roman*)}}

\begin{document}

\title{Measure-geometric Laplacians on the real line}

\author[M.\ Kesseb\"ohmer]{M.\ Kesseb\"ohmer}
\address[M.\ Kesseb\"ohmer and H.~Weyer]{FB 3 -- Mathematik, Universit\"at Bremen, Bibliothekstr. 1, 28359 Bremen, Germany}
\author[T.\ Samuel]{T.\ Samuel}
\address[T.\ Samuel]{School of Mathematics, University of Birmingham, Edgbaston, Birmingham, B15 2TT, UK}
\author{H.~Weyer}


\maketitle

\begin{abstract}
Motivated by the fundamental theorem of calculus, and based on the works of Feller as well as Kac and Kre\u{\i}n, given an atomless Borel probability measure $\eta$ supported on a compact subset of $\R$, Freiberg and Z\"{a}hle introduced a measure-geometric approach to define a first order differential operator $\nabla_{\eta}$ and a second order differential operator $\Delta_{\eta}$, with respect to $\eta$.  We generalise this approach to measures of the form $\eta  \coloneqq \nu + \delta$, where $\nu$ is continuous and $\delta$ is finitely supported.  We determine analytic properties of $\nabla_{\eta}$ and $\Delta_{\eta}$ and show that $\Delta_{\eta}$ is a densely defined, unbounded, linear, self-adjoint operator with compact resolvent.  Moreover, we give a systematic way to calculate  the eigenvalues and eigenfunctions of  $\Delta_{\eta}$.  For two leading examples, we determine the eigenvalues and the eigenfunctions, as well as the asymptotic growth rates of the eigenvalue counting function.
%
%We generalise the measure-geometric approach to define a first order differential operator $\nabla_{\eta}$ and a second order differential operator  $\Delta_{\eta}$ for measures of the form $\eta  \coloneqq \nu + \delta$, where $\nu$ is continuous and $\delta$ is finitely supported; thus extending the program developed by Freiberg and Z{\"a}hle.  We determine analytic properties of $\nabla_{\eta}$ and $\Delta_{\eta}$ and show that $\Delta_{\eta}$ is a densely defined, unbounded, linear, self-adjoint operator with compact resolvent.  Moreover, we give a systematic way to calculate  the eigenvalues and eigenfunctions of  $\Delta_{\eta}$.  For two leading examples, we determine the eigenvalues and the eigenfunctions, as well as the asymptotic growth rates of the eigenvalue counting function.
%
\end{abstract}

\section{Introduction and statement of main results}

Kac posed the following famous question in \cite{MR0201237}. ``Can one hear the shape of a drum?''  Namely, can one reconstruct the geometry of a $n$-dimensional manifold from the eigenvalues of the associated Laplacian. In  1964 Milnor \cite{MR0162204} showed the existence of a pair of 16-dimensional tori whose associated Laplacians have the same eigenvalues but  which have different shapes.  Subsequently, for a given $n \geq 4$, Urakawa \cite{MR690649} produced the first examples of domains in $\R^{n}$ with this property. The problem in two dimensions remained open until 1992, when Gordon, Webb, and Wolpert \cite{MR1181812} constructed a pair of regions in the plane that have different shapes but whose associated Laplacians have identical eigenvalues.  Nevertheless, as observed by Weyl \cite{Weyl:1912}, Berry \cite{MR556688,MR573427}, Lapidus  \textsl{et al.\ }\cite{MR994168,MR1234502,MR1046509,MR1189091,MR1356166}, Beals and Greiner \cite{MR2513598} and many others, the spectrum of a Laplacian still tells us a lot about the shape of the underlying geometric structure.

As a special case of the generalised Kre\u{\i}n-Feller operator \cite{Fe57,KK68}, in \cite{FZ02} Freiberg and Z\"ahle introduced a measure-geometric approach to define a first order differential operator $\nabla_{\eta}$ and a second order differential operator $\Delta_{\eta, \eta} \coloneqq \nabla_{\eta} \circ \nabla_{\eta}$, with respect to an atomless Borel probability measure $\eta$ supported on a compact subset of $\mathbb{R}$.
%
%Motivated by the fundamental theorem of calculus, and based on the works of Feller \cite{Fe57} as well as Kac and Kre\u{\i}n \cite{KK68}, given an atomless Borel probability measure $\eta$ supported on a compact subset of $\R$, Freiberg and Z\"{a}hle \cite{FZ02} introduced a measure-geometric approach to define a first order differential operator $\nabla_{\eta}$ and a second order differential operator $\Delta_{\eta, \eta} \coloneqq \nabla_{\eta} \circ \nabla_{\eta}$, with respect to $\eta$.  
%
In the case that $\eta$ is the Lebesgue measure, it was shown that $\nabla_{\eta}$ coincides with the weak derivative.  Moreover, a harmonic calculus for $\Delta_{\eta, \eta}$ was developed and, when $\eta$ is a self-similar measure supported on a Cantor set, Freiberg and Z\"{a}hle proved that the eigenvalue counting function of $\Delta_{\eta, \eta}$ is comparable to the square-root function. In \cite{KSW16} the exact eigenvalues and eigenfunctions of $\Delta_{\eta, \eta}$ were obtained and it was shown that the eigenvalues do not depend on the given measure. Arzt \cite{A15b}, Freiberg \cite{MR2017701,MR2030736,Fr05}, Fujita \cite{Fu87}, and Kotani and Watanabe \cite{MR661628} have also considered the Kre\u{\i}n-Feller operator $\Delta_{\eta, \Lambda} \coloneqq \nabla_{\eta} \circ \nabla_{\Lambda}$, where $\eta$ denotes a continuous Borel probability measure and $\Lambda$ denotes the Lebesgue measure. In the case that $\eta$ is a purely atomic measure, it has been shown in \cite{MR2513598} that the eigenvalues of $\Delta_{\eta, \Lambda}$  highly depend on the position and weights of the atoms.  Interestingly, the operators $\Delta_{\eta, \Lambda}$ and $\Delta_{\eta, \eta}$, in the case that $\eta$ is continuous, also appear as the infinitesimal generator of Liouville Brownian motion \cite{X_Jin_2017} and Liouville quantum gravity \cite{MR3272822}.

In \cite{KSW17} it was shown that the framework of Freiberg and Z\"{a}hle can be extended to include purely atomic measures $\eta$. Unlike in the case when one has a measure with a continuous distribution function, it was proven that the operators $\nabla_{\eta}$ and $\Delta_{\eta, \eta}$ are no longer symmetric. To circumvent this problem, the $\eta$-Laplacian was defined by $\Delta_{\eta} = -\nabla_{\eta}^{*} \circ \nabla_{\eta}$, where $\nabla_{\eta}^{*}$ denotes the adjoint of $\nabla_{\eta}$. Further, a matrix representation for these operators was given and shown to coincide with the normalised graph Laplacian of a cycle graph \cite{Bi93}. Moreover, the eigenvalues of $\Delta_{\eta}$ depend only on the weights of the atoms and are independent of the positions of the atoms.

In this article, we continue to develop the program of Freiberg and Z\"{a}hle for measures of the form $\eta = \nu + \delta$, where $\nu$ is continuous and $\delta$ is a finite sum of weighted Dirac point masses.   Indeed, we show, for such an $\eta$, one can define a first and a second order measure geometric differential operator, $\nabla_{\eta}$ and $\Delta_{\eta} = -\nabla_{\eta}^{*} \circ \nabla_{\eta}$ respectively.   Moreover, we determine properties of both operators; in particular, we show the following.

\begin{thm}\label{thm:main1}
The operator $\Delta_{\eta}$ is densely defined on the space $L^{2}_{\eta}$ of square-$\eta$-integrable functions.  Further, it is linear, self-adjoint, non-positive and has compact resolvent.
\end{thm}

For ease of exposition in the following result we assume that $\nu$ is a probability measure on $(0,1]$ with distribution function $F_{\nu}$ and let $\delta = \sum_{i=1}^{N} \alpha_{i} \delta_{z_{i}}$ with $0 < z_{1} < \ldots < z_{N} = 1$.

\begin{thm}\label{thm:main2}
A square-$\eta$-integrable function $f$ is an eigenfunction of $\Delta_{\eta}$ with corresponding eigenvalue $\lambda$ if and only if it is of the form
\begin{align*}
f(x) =
\begin{cases}
a_{1} \sin(b F_{\nu}(x) + \gamma_{1}) & \text{if} \; x \in (0, z_{1}],\\
\hspace{3em} \vdots & \hspace{2.5em} \vdots \\
a_{N} \sin(b F_{\nu}(x) + \gamma_{N}) & \text{if} \; x \in (z_{N-1}, 1],
\end{cases}
\end{align*}
and $\lambda = - b^{2}$, where $b, a_{1}, \ldots, a_{N}, \gamma_{1},\ldots, \gamma_{N} \in \R$ satisfy the following system of equations:
\begin{align*}
\begin{aligned}
&\alpha_{j} b a_{j+1} \cos(b F_{\nu}(z_{j}) +\gamma_{j+1}) = a_{j+1} \sin(b F_{\nu}(z_{j}) +\gamma_{j+1}) - a_{j} \sin(b F_{\nu}(z_{j}) +\gamma_{j}), \\
&\alpha_{j} b^2 a_{j} \sin(b F_{\nu}(z_{j}) +\gamma_{j}) =  a_{j} b \cos(b F_{\nu}(z_{j}) +\gamma_{j}) - a_{j+1} b \cos(b F_{\nu}(z_{j}) +\gamma_{j+1}),
\end{aligned}
\end{align*}
for $j \in \{1,\ldots,N-1\}$, and
\begin{align*}
\begin{aligned}
&\alpha_{N} b a_{1} \cos(\gamma_{1}) = a_{1} \sin(\gamma_{1}) - a_{N} \sin(b +\gamma_{N}), \\
&\alpha_{N} b^2 a_{N} \sin(b +\gamma_{N}) = a_{N} b \cos(b +\gamma_{N}) - a_{1} b \cos(\gamma_{1}).
\end{aligned}
\end{align*}
\end{thm}

\begin{rem}\label{rmk:KSW16}
To prove \Cref{thm:main1} , it is sufficient to consider the case when the continuous part $\nu$ of $\eta$  restricted to the interval between two consecutive atoms is either  zero  or Lebesgue, and to prove \Cref{thm:main2}, it is sufficient to consider the case when $\nu = \Lambda$.  Indeed, the general case follows by appropriately composing the operator with the distribution function $F_{\nu}$ as in \cite{KSW16}.
\end{rem}

\Cref{thm:main2} shows that the eigenvalues depend on the weights of the Dirac point masses and their positions relative to $\nu$, but that they are independent of the distribution of $\nu$; this condition is different than that given for the Kre\u{\i}n-Feller operator $\Delta_{\eta, \Lambda}$, where $\eta$  is a purely atomic measure, compare with \cite{MR2513598}.

Further, we investigate two leading examples in detail and determine their eigenvalues and eigenfunctions explicitly. In contrast to the classical theory and to the case of atomless measures we see that all eigenspaces are one-dimensional; however, the asymptotic growth rate of the eigenvalue counting function $N_{\eta}$ behaves as in the classical situation, namely
\begin{align*}
\displaystyle \lim_{x \to \infty} \frac{\pi N_{\eta}(x)}{\sqrt{x}} = 1.
\end{align*}  

This article is structured as follows. In \Cref{sec:general} we define the operators $\nabla_{\eta}$, $\nabla_{\eta}^{*}$ and $\Delta_{\eta}$ and prove that $\Delta_{\eta}$ is a densely defined self-adjoint operator on $L^{2}_{\mu}$.  Further, we show that $\nabla_{\eta}$ and $\nabla_{\eta}^{*}$ are closed, give an explicit description of their domains and ranges, and prove that $\Delta_{\eta}$ has compact resolvent.  From this and \Cref{rmk:KSW16} we conclude the proof of \Cref{thm:main1}.  In \Cref{sec:eigenvalues_functions} we determine spectral properties of $\Delta_{\eta}$.  We divide this section into three parts. In the first part (\Cref{sec:3_1}), we give a system of equations which allows one to obtain the eigenvalues and find a general form of the eigenfunctions, hence proving \Cref{thm:main2}. In the second part (\Cref{sec:N=1}), we solve the system of equations given in \Cref{thm:main2} for the case that $N = 1$ and illustrate the results in an example. The third part (\Cref{sec:3-2}) deals with the case when $N = 2$ and when the Dirac point masses are uniformly distributed and equally weighted.  We end this final section with an example which illustrates our results in this latter setting.

\section{The operators $\nabla_{\eta}$ and $\Delta_{\eta}$}\label{sec:general}

Let $\eta$ denote a finite Borel measure on $\R$ and let $a, b \in \R$ be such that the convex hull of $\mathrm{supp}(\eta)$ is equal to $[a, b]$.  Here $\mathrm{supp}(\eta)$ denotes the support of $\eta$, that is, the smallest compact set with full measure.  We assume that $\eta( \{ a \} ) = 0$ and set $M = (a, b]$.  For $K \subseteq M$, we let $\eta\vert_{K}$ be the restriction of $\eta$ to the set $K$, that is, $\eta\vert_{K}(A) \coloneqq \eta(A \cap K)$ for all Borel sets $A \subseteq \R$; the same notation is used for functions.  When it is clear from context, we write $\eta$ for $\eta\vert_{K}$.  We denote the set of real-valued square-$\eta$-integrable functions with domain equal to $M$ by $\mathfrak{L}^{2}_{\eta}$, define $\mathcal{N}_{\eta}$ to be the set of $\mathfrak{L}^{2}_{\eta}$-functions which are constant zero $\eta$-almost everywhere, and let $L^{2}_{\eta} \coloneqq \mathfrak{L}^{2}_{\eta} / \mathcal{N}_{\eta}$. Following convention, when we write $f \in L^{2}_{\eta}$, we mean that there exists an equivalence class of $L^{2}_{\eta}$ to which $f$ belongs. When it is clear from context, we will use the same notation for a function $f \in \mathfrak{L}^{2}_{\eta}$ and for the equivalence class in $L^{2}_{\eta}$ to which it belongs.  We equip $L^{2}_{\eta}$ with the inner product given by 
\begin{align*}
\langle f, g \rangle_{\eta} \coloneqq \int  f  g \, \mathrm{d}\eta.
\end{align*}
We denote by $\lVert \cdot \rVert_{\eta}$ the associate $L^{2}_{\eta}$-norm.  For $d \in \R$ we set $d \Z \coloneqq \{ d k \colon k \in \Z \}$ and let $\delta_{d \Z}$ denote the Dirac comb $\delta_{d \Z} \coloneqq \sum_{k \in \Z} \delta_{d k}$.  Here, for $z \in \R$, we write $\delta_z$ for the Dirac point mass at $z$.  For a function $f \colon M \to \R$ we let $\mathbf{f} \colon \R \to \R$ be the periodic extension of $f$, that is $\mathbf{f}(x) = f(x)$ for all $x \in M$ and $\mathbf{f}(x) = \mathbf{f}(x + (b-a) k)$ for all $x \in \R$ and $k \in \Z$.  We define the set $\mathscr{D}_{\eta}^{1}$ of \textsl{$\eta$-differentiable functions} by
\begin{align}\label{Def:D1line}
\mathscr{D}_{\eta}^{1} \coloneqq \left\{ f \in \mathfrak{L}^2_{\eta} \colon \text{there exists} \; f' \in L^2_{\eta}  \; \text{with} \; \mathbf{f}(x) = \mathbf{f}(y) + \int \mathds{1}_{[y,x)} \mathbf{f}^{\prime} \, \mathrm{d}\eta * \delta_{(b-a)\Z} \; \text{for all} \; x, y \in \R \; \text{with} \; y < x \right\}.
\end{align}
Here $\eta * \delta_{(b-a)\Z}$ denotes the convolution of the measure $\eta$ and the Dirac comb $\delta_{(b-a)\Z}$; see, for instance, \cite{Halmos:1950} for the definition of and results on convolutions of measures.  

If $f \in \mathscr{D}_{\eta}^{1}$, then $f$ is left-continuous with discontinuities occurring only in a subset of $\{ z_{1}, \dots, z_{N} \}$.  As the function $f'$ defined in \eqref{Def:D1line} is unique in $L^2_{\eta}$, the operator $\nabla_{\eta} \colon \mathscr{D}_{\eta}^{1} \to L^2_{\eta}$ given by $\nabla_{\eta}f \coloneqq f'$ is well-defined and called the \textsl{${\eta}$-derivative}.  By the linearity of the integral equation in \eqref{Def:D1line}, it follows that $\nabla_{\eta}$ is linear.  Additionally, if $f, g \in \mathscr{D}_{\eta}^{1}$ with $f \neq g$, then $\lVert f - g \rVert_{\eta} \neq 0$.  Thus, we may view $\mathscr{D}_{\eta}^{1}$ as a collection of real-valued square-$\eta$-integrable functions, or as a collection of equivalence classes of $L^{2}_{\eta}$ and in the latter setting, we define $\nabla_{\eta}$ accordingly; namely, if $f \in \mathscr{D}_{\eta}^{1}$, then $\nabla_{\eta}$ maps the equivalence class of $f$ to the equivalence class of $\nabla_{\eta}f$.

\begin{rem}\label{rmk:dense_inclusions_functions_spaces}
Let $C(M)$ denote the set of continuous functions $f \colon M \to \R$ and let $\mathcal{C}^1(M)$ denote the set of $f  \in \mathcal{C}(M)$ such that $f$ is differentiable on $(a, b)$ and left-differentiable at $b$. In the case that $\eta$ is a continuous  Borel measure, the set $\mathscr{D}_{\eta}^{1}$ given in \eqref{Def:D1line} is contained in $C(M)$ and equal to the set $\mathscr{D}^{\eta}_{1}$ given in \cite{FZ02}, which, if $\eta=\Lambda$, is in turn equal to the Sobolev space $W_{2}^{1}$.  Moreover, an application of the fundamental theorem of calculus yields that $\mathcal{C}^{1}(M)$ is contained in $\mathscr{D}_{\Lambda}^{1}$. Thus, we have $\mathcal{C}^1(M) \subseteq  \mathscr{D}_{\Lambda}^{1} \subseteq \mathcal{C}(M) \subseteq L_{\Lambda}^{2}$, where each set is dense with respect to $\lVert \cdot \rVert_{\Lambda}$, in the succeeding one.
\end{rem}

\begin{rem}\label{rmk:KSW16_2}
By \Cref{rmk:KSW16}, and a rescaling and translation argument, it is sufficient to prove \Cref{thm:main1} under the following assumptions.  There exists $N \in \N$ and $\underline{c} = (c_{1}, \dots, c_{N}) \in \{ 0, 1\}^{N}$ with $\underline{c} \neq \underline{0}$, such that 
\begin{enumerate}
\item $a = 0$ and $b = 1$, in which case $M = \I \coloneqq (0,1]$, and
\item $\eta = \Gamma + \sum_{i=1}^{N}\alpha_i \delta_{z_i}$, where $0 < z_{1} < \cdots < z_{N} = 1$ and $\mathrm{d}\Gamma = (\sum_{i=1}^{N} c_{i} \mathds{1}_{(z_{i-1}, z_{i}]}) \mathrm{d}\Lambda$.
\end{enumerate}
Thus, throughout this section we assume that $\eta$ has this form.
\end{rem}

For convenience, we  set $z_{N + 1}\coloneqq1 + z_{1}$ and $A_{i} \coloneqq (z_{i-1},z_{i}]$, for $i \in \{ 1, \ldots, N \}$. Given a bounded interval $J$, following convention, we let $J^{\textup{o}}$ denote the interior of $J$. Further, we let $\mathds{1}_{J}$ denote the characteristic function on $J$, and in the case that $J = \I$, we write $\mathds{1}$ for $\mathds{1}_{J}$.

As with the classical weak Laplacian, $\nabla_{\eta}f$ reflects local properties of $f$.  Indeed, we have, for  $i \in \{ 1, \ldots, N \}$,
\begin{align}\label{eq:derivative_calculation}
\nabla_{\eta} f(z_i) = \lim_{\varepsilon \searrow 0} \frac{\mathbf{f}(z_i + \varepsilon) - \mathbf{f}(z_i)}{\alpha_i}
\quad \text{and} \quad
\nabla_{\eta} f(x) = f_{\Gamma}^{\prime} (x),
\end{align}
where $x \in A_{i}^{\textup{o}}$ and $f_{\Gamma}^{\prime} \in L^2_{\Gamma}$ is such that
\begin{align*}
\mathbf{f}(x) =  \lim_{\varepsilon \searrow 0} \mathbf{f}(z_i + \varepsilon) + \int \mathds{1}_{[z_i, x)} \mathbf{f}_{\Gamma}^{\prime} \, \mathrm{d}{\Gamma}.
\end{align*}
If $c_{i}=1$, then $f_{\Gamma}^{\prime} \rvert_{A_{i}^{\textup{o}}}$ coincides with the weak derivative, and if $c_{i} = 0$, then $f_{\Gamma}^{\prime} \rvert_{A_{i}^{\textup{o}}}$ can be chosen arbitrarily.
\begin{prop}\label{prop:D1_dense_in_L2}
The set $\mathscr{D}_{\eta}^{1}$ is dense in $L_{\eta}^2$ with respect to $\lVert \cdot \rVert_{\eta}$.
\end{prop}
\begin{proof}
For $i \in \{1,\ldots,N\}$ we show, for given $f \in L_{\Lambda}^{2}$ and $\varepsilon > 0$, that there exists $g \in \mathscr{D}_{\eta}^{1}$ with $\lVert f - g \rVert_{\eta} < \varepsilon$. Following this, for $i \in \{1,\ldots,N\}$, we construct a sequence $(h_{n})_{n \in \N}$ in $\mathscr{D}_{\eta}^{1}$ with $\lim_{n \to \infty} \lVert h_{n} - \mathds{1}_{\{ z_{i} \}} \rVert_{\eta} = 0$.
 
Let us begin by showing the first statement for a fixed $i \in \{1, \ldots, N-1 \}$; the case $i=N$ follows by a similar argument and is left to the reader.  For $0 < \varepsilon < (z_{i+1} - z_{i})/2$, $n \in \N_{0}$ and $x \in (0,1]$ set $c_{\varepsilon, n} \coloneqq \pi n/(z_{i+1}-z_{i}-2\varepsilon)$ and $f_{\varepsilon,n}(x) \coloneqq \int_{0}^{x}f^{\prime}_{\varepsilon,n}  \, \mathrm{d}{\Lambda}$, where  
\begin{align*}
f^{\prime}_{\varepsilon,n}(x) &\coloneqq \begin{cases}
-(32 c_{\varepsilon, n}/\varepsilon) \left( x - \left( z_{i} + \varepsilon/4 \right) \right) &\text{if} \; x \in ( z_{i} +\varepsilon/4, z_{i} + 3\varepsilon/8 ],\\
(32 c_{\varepsilon, n}/\varepsilon) \left( x - \left( z_{i} + \varepsilon/2 \right) \right) &\text{if} \; x \in ( z_{i} + 3\varepsilon/8, z_{i} + \varepsilon/2 ],\\
-c_{\varepsilon, n} \cos \left( (2 \pi/\varepsilon) \, \left(x - \left(z_{i} + \varepsilon/2 \right) \right)\right) + c_{\varepsilon, n} &\text{if} \; x \in (z_{i} + \varepsilon/2, z_{i} + \varepsilon ],\\
2 c_{\varepsilon, n} \cos \left( 2 c_{\varepsilon, n} (x - (z_{i} +\varepsilon) ) \right) &\text{if} \; x \in ( z_{i} + \varepsilon, z_{i+1} - \varepsilon ],\\
c_{\varepsilon, n}  \cos \left( (2 \pi/\varepsilon) \, \left(x - \left(z_{i + 1} - \varepsilon \right) \right) \right) + c_{\varepsilon, n} &\text{if} \; x \in (z_{i+1} - \varepsilon, z_{i+1} - \varepsilon/2 ],\\
-(32 c_{\varepsilon, n}/\varepsilon) \left( x - \left( z_{i+1} - \varepsilon/2 \right) \right) &\text{if} \; x \in ( z_{i+1} - \varepsilon/2, z_{i+1} - 3\varepsilon/8 ],\\
(32c_{\varepsilon, n}/\varepsilon) \left( x - \left( z_{i+1} - \varepsilon/4 \right) \right) &\text{if} \; x \in ( z_{i+1} - 3\varepsilon/8, z_{i+1} - \varepsilon/4 ], \\
0 &\text{otherwise}.
\end{cases}
\intertext{By definition, $f_{\varepsilon,n} \in \mathscr{D}_{\eta}^{1}$, $f_{\varepsilon,n}(x) = 0$, for $x \in \I\setminus A_{i}^{o}$, and $f_{\varepsilon,n}(x) = \sin \left(2c_{\varepsilon, n} (x - (z_{i} +\varepsilon) ) \right)$, for $x \in (z_{i} + \varepsilon, z_{i+1} - \varepsilon)$. Further, set}
g_{\varepsilon,n} (x) &\coloneqq \begin{cases}
-(1/2) \cos \left( (2 \pi/\varepsilon) \left(x - \left(z_{i} +\varepsilon/2 \right) \right)  \right) + 1/2 &\text{if} \; x \in (z_{i} + \varepsilon/2, z_{i} + \varepsilon ],\\
\cos \left( 2c_{\varepsilon, n} (x - (z_{i} +\varepsilon) ) \right) &\text{if} \; x \in ( z_{i} + \varepsilon, z_{i+1} - \varepsilon ], \\
(1/2) \cos \left( (2 \pi/\varepsilon) \, (x - (z_{i+1} - \varepsilon))  \right) + 1/2 &\text{if} \; x \in ( z_{i+1} - \varepsilon, z_{i+1} - \varepsilon/2 ],\\
0 &  \text{otherwise}.
\end{cases}
\end{align*}
Notice, $g_{\varepsilon,n} \in \mathscr{D}_{\eta}^{1}$, $g_{\varepsilon,n}(x) = 0$, for $x \in \I\setminus A_{i}^{o}$, and $g_{\varepsilon,n}(x) = \cos \left( 2c_{\varepsilon, n} (x - (z_{i} +\varepsilon) ) \right)$, for $x \in (z_{i} + \varepsilon, z_{i+1} - \varepsilon)$.

The functions $f_{\varepsilon,n}$ and $g_{\varepsilon,n}$ are differentiable (in the classical sense) and hence, by \Cref{rmk:dense_inclusions_functions_spaces}, lie in $\mathscr{D}_{\eta\vert_{A_{i}}}^{1}$. For $i \in \{ 1, \ldots, N\}$, define $\mathcal{F}_{i} \coloneqq \{ f_{\varepsilon,n}, g_{\varepsilon,n} \colon 0 < \varepsilon < (z_{i+1} - z_{i})/2 \; \text{and} \; n \in \N_{0} \}$. From the fact that the set $\{ x \mapsto \sin(2 \pi n x/\Lambda(A_{i})) \colon n \in \N \} \cup \{ x \mapsto \cos(2 \pi n x/ \Lambda(A_{i})) \colon n \in \N_{0} \}$ forms a basis of $L_{\Lambda\rvert_{A_{i}}}^{2}$, it follows that the span of $\bigcup_{i=1}^{N} \mathcal{F}_{i} \subseteq \mathscr{D}_{\eta}^{1}$ is a dense subset of $L_{\Lambda}^2$ with respect to $\lVert \cdot \rVert_{\Lambda}$. 

For the second statement, again let $i \in \{ 1, \ldots, N-1\}$ be fixed; the case $i=N$ follows by a similar argument and is left to the reader. Let $K$ be the smallest natural number with $1/K < \min\{ z_{i} - z_{i-1} \colon i \in \{1,\ldots,N\} \}$.  We divide the proof into four cases, namely when $c_{i}$ and $c_{i+1}$ are zero or one.

In the case that the atom is isolated, that is $c_{i} = c_{i+1} = 0$, the function $h \coloneqq \mathds{1}_{A_{i}}$ lies in $\mathscr{D}_{\eta}^{1}$ and $\displaystyle \lVert h - \mathds{1}_{\{ z_{i} \}} \rVert_{\eta} = 0$. For the cases when $c_{i+1} = 1$, we set, for $n \in \N$ with $n \geq K$,
\begin{align*}
h_{n}(x) &\coloneqq \begin{cases}
(1-\cos(n\pi(x-(z_{i} -1/n)))) / 2 &
\parbox{16em}{if $c_{i} = 1$ and $x \in (z_{i} - 1/n,z_{i} + 1/n )$, or\\
if $c_{i} = 0$ and $x \in (z_{i},z_{i} + 1/n )$,
}\\[0.5em]
1 &\text{if} \; c_{i} = 0 \; \text{and} \; x \in (z_{i-1} ,z_{i}],\\
0 &\text{otherwise}.
\end{cases}
\end{align*}
Observe that $\displaystyle \lim_{n \to \infty} \lVert h_{n} - \mathds{1}_{\{ z_{i} \}} \rVert_{\eta} = 0$ and moreover, that $h_{n} \in \mathscr{D}_{\eta}^{1}$.

Finally, we consider the case $c_{i} = 1$ and $c_{i+1} = 0$.  For this let $j \in \{i + 2, \ldots, N \}$ be the smallest such integer with $c_{j} = 1$. If no such $j$ exists then let $j \in \{1, \ldots, i \}$ be the smallest such integer with $c_{j} = 1$.  Observe that in the first case it is sufficient to approximate $\mathds{1}_{\{z_{i}, \dots, z_{j}\}}$ and, in the second case, $\mathds{1}_{\{z_{1}, \dots, z_{j}, z_{i}, \dots, z_{N}\}}$.  Here we prove the former of these two cases as the latter follows analogously. For $n \in \N$ with $n \geq K$, set
\begin{align*}
h_{n}(x) &\coloneqq \begin{cases}
(1-\cos(n\pi(x-(z_{i} - 1/n)))) / 2 &\text{if} \; x \in (z_{i} - 1/n, z_{i} ],\\
(1-\cos(n\pi(x-(z_{j} +1/n)))) / 2 &\text{if} \; x \in (z_{j},z_{j} + 1/n ],\\
1 &\text{if} \; x \in (z_{i} ,z_{j}],\\
0 &\text{otherwise}.
\end{cases}
\end{align*}
By definition, we have $\displaystyle \lim_{n \to \infty} \lVert h_{n} - \mathds{1}_{\{ z_{i}, \ldots, z_{j} \}} \rVert_{\eta} = 0$ and $h_{n} \in \mathscr{D}_{\eta}^{1}$.
\end{proof}

From \eqref{Def:D1line}, if $f \in \mathscr{D}_{\eta}^{1}$, then
\begin{align}\label{eq:int_deriv_0}
\langle \nabla_{\eta} f, \mathds{1} \rangle_{\eta} = \int \nabla_{\eta}f \, \mathrm{d}\eta = 0,
\end{align}
which implies that the zero function is the only constant function which can occur as an $\eta$-derivative. 

Letting $\varrho$ denote the (natural) quotient map from $L_{\eta}^{2}$ to $L_{\eta}^{2} / \{ c \mathds{1} \colon c \in \R \}$, in the following proposition we show that the image of the range of $\nabla_{\eta}$ under $\varrho$ is dense in $L_{\eta}^{2} / \{ c \mathds{1} \colon c \in \R \}$.  By the continuity of the inner product $\langle \cdot, \cdot \rangle_{\eta}$, this implies that the orthogonal complement of the range of $\nabla_{\eta}$ is equal to the set of constant functions on $\I$.

\begin{prop}\label{prop:range_nabla_eta}
The image under $\varrho$ of the range of $\nabla_{\eta}$ is dense in the quotient space $L_{\eta}^{2} / \{ c \mathds{1} \colon c \in \R \}$.
\end{prop}

\begin{proof}
The set of bounded continuous functions $\mathcal{C}_{B}(\I)$ with domain equal to $\I$ is dense in $L_{\eta}^{2}$, and so the image of $E \coloneqq \{ g \in \mathcal{C}_{B}(\I) \colon \langle g, \mathds{1} \rangle_{\eta} = 0 \}$ under $\varrho$ is dense in $L_{\eta}^{2} / \{ c \mathds{1} \colon c \in \R \}$.  For $g \in E$, setting $f(x) = \langle g, \mathds{1}_{[0,x)} \rangle_{\eta}$ for  $x \in \I$, observe that $f$ is left-continuous and bounded, and so $f \in L_{\eta}^{2}$.  Since $g \in E$, by \eqref{eq:int_deriv_0}, we have $f \in \mathscr{D}_{\eta}^{1}$, where $\nabla_{\eta} f = g$.  In other words, $E$ is contained in the range of $\nabla_{\eta}$ and hence the image of the range of $\nabla_{\eta}$ under $\varrho$ is dense in $L_{\eta}^{2} / \{ c \mathds{1} \colon c \in \R \}$.
\end{proof}

As in \cite{KSW17, KL01}, we use a Dirichlet form $\E_{\eta}$ to define the measure geometric Laplacian $\Delta_{\eta}$.  For this we use the following properties of unbounded operators; see for instance \cite{RS81} for further details. The \textsl{graph of a densely defined linear operator $T$} on $L^{2}_{\eta}$ is $\Gamma(T) \coloneqq \{ (f, T(f)) \in L^{2}_{\eta} \times L^{2}_{\eta} \colon f \in \textup{Dom}(T) \}$, where $\textup{Dom}(T)$ denotes the domain of $T$.  In the case that $\Gamma(T)$ is closed in $L^{2}_{\eta} \times L^{2}_{\eta}$, we say that $T$ is \textsl{closed}. If $T_1$ is a densely defined operator on $L^{2}_{\eta}$ and if $\Gamma(T_1) \supseteq \Gamma(T)$, then $T_1$ is called an \textsl{extension of $T$}. When $T$ has a closed extension, $T$ is said to be \textsl{closable}.  The smallest closed extension of $T$, denoted by $\overline{T}$, is the \textsl{closure} of $T$.

For a densely defined operator $T$ on $L^{2}_{\eta}$ we let $\textup{Dom}(T^*)$ be the set of $f \in L^{2}_{\eta}$ for which there exists $h \in L^{2}_{\eta}$ with $\langle T(g), f \rangle_{\eta} = \langle g, h \rangle_{\eta}$ for all $g \in \textup{Dom}(T)$.  For each such $f \in \textup{Dom}(T^*)$, we define $T^{*}(f) \coloneqq h$.  We refer to $T^{*}$ as the \textsl{adjoint} of $T$.  We call $T$ \textsl{symmetric} if $\textup{Dom}(T) \subseteq \textup{Dom}(T^{*})$ and $T(f) = T^{*}(f)$, for all $f \in \textup{Dom}(T)$.  Equivalently, $T$ is symmetric if and only if $\langle T(f), g \rangle_{\eta} = \langle f, T(g) \rangle_{\eta}$ for all $f, g \in \textup{Dom}(T)$.  If in addition to $T$ being symmetric, we have that $\textup{Dom}(T) = \textup{Dom}(T^*)$, then we say that $T$ is \textsl{self-adjoint}.

\begin{thm}[\cite{RS81}]\label{thm:Reed_Simon_VIII.3}
If $T$ is an unbounded, densely defined operator on $L^{2}_{\eta}$, then the following holds.
\begin{enumerate}
\item The operator $T^*$ is closed.
\item The operator $T$ is closable if and only if $\textup{Dom}(T^*)$ is dense in $L_{\eta}^{2}$ in which case $\overline{T} = T^{**}$.
\item If $T$ is closable, then $( \overline{T} )^* = T^*$.
\end{enumerate}
\end{thm}

In the following proposition we show that the domain of $\nabla_{\eta}^{*}$ is equal to
\begin{align}\label{eq:adjoint_domain}
\mathscr{D}_{\eta}^{1*} \coloneqq \left\{ f \in \mathfrak{L}^2_{\eta} \colon \text{there exists} \; f^{*} \in L^2_{\eta}  \; \text{with} \; \mathbf{f}(x) = \mathbf{f}(y) + \int \mathds{1}_{(x, y]} \mathbf{f}^{*} \, \mathrm{d}\eta * \delta_{\Z} \; \text{for all} \; x, y \in \R \; \text{with} \; x < y \right\}.
\end{align}
Notice, if $f \in \mathscr{D}_{\eta}^{1*}$, then $f$ is right-continuous with discontinuities occurring only at points in a subset of $\{z_1, \ldots, z_N \}$.  Moreover, if $f, g \in \mathscr{D}_{\eta}^{1*}$ with $f \neq g$, then $\lVert f - g \rVert_{\eta} \neq 0$.  Thus, as with $\mathscr{D}_{\eta}^{1}$, we may view $\mathscr{D}_{\eta}^{1*}$ as a collection of real-valued square-$\eta$-integrable functions, or as a collection of equivalence classes of $L^{2}_{\eta}$.

\begin{prop}\label{lem:d1mu*}
The domain of $\nabla_{\eta}^{*}$ is equal to $\mathscr{D}_{\eta}^{1*}$.
\end{prop}

\begin{proof}
Let $f \in \textup{Dom}(\nabla_{\eta}^{*})$.  Using the fact that $ \langle \nabla_{\eta}^{*} f, \mathds{1} \rangle_{\eta} = \langle f, \nabla_{\eta} \mathds{1} \rangle_{\eta} = \langle f, 0 \rangle_{\eta} = 0$ and Fubini's Theorem , we obtain, for $g \in \mathscr{D}_{\eta}^{1}$,
\begin{align*}
\int \mathds{1}_{(0, 1]} f (\nabla_{\eta} g) \ \mathrm{d} \eta 
&= \int \mathds{1}_{(0, 1]} (\nabla_{\eta}^{*} f) g \ \mathrm{d} \eta\\
&= \int \mathds{1}_{(0, 1]}(x) (\nabla_{\eta}^{*} f)(x) \left( \mathbf{g}(0) + \int \mathds{1}_{[0, x)}(y) (\nabla_{\eta} g)(y) \ \mathrm{d} \eta(y) \right) \ \mathrm{d} \eta(x)\\
&= \int \mathds{1}_{[0, 1)}(y) \nabla_{\eta} g(y) \int \mathds{1}_{(y, 1]}(x) (\nabla_{\eta}^{*}f)(x) \ \mathrm{d} \eta(x) \ \mathrm{d} \eta(y).
\end{align*}
This, together with \Cref{prop:range_nabla_eta} and the fact that $\langle \mathds{1}, \nabla_{\eta} g \rangle_{\eta} = 0$, implies $\varrho (f - \int \mathds{1}_{(\cdot ,1]} \nabla_{\eta}^{*} f \ \mathrm{d}\eta) = 0$. Hence, there exists $c \in \R$ so that, for $\eta$-almost all $y \in (0, 1]$,
\begin{align*}
f(y) - \int \mathds{1}_{(y ,1]} \nabla_{\eta}^{*} f \ \mathrm{d}\eta = c.
\end{align*}
This yields that $f$ is a right-continuous function with $c = \mathbf{f}(0)$. When setting $f^{*} \coloneqq  \nabla_{\eta}^{*} f$, since $\mathbf{f}(0) = f(1)$ and $\langle \nabla_{\eta}^{*} f, \mathds{1} \rangle_{\eta} = 0$, we have that $\mathbf{f}^*$ fulfils the integral equation in \eqref{eq:adjoint_domain} and therefore, $\textup{Dom}(\nabla_{\eta}^{*}) \subseteq \mathscr{D}_{\eta}^{1*}$.

For the reverse inclusion, let $g \in \textup{Dom}(\nabla_{\eta})$, $f \in \mathscr{D}_{\eta}^{1*}$ and $f^{*}$ be as in \eqref{eq:adjoint_domain}. By Fubini's Theorem and the fact that $\langle f^{*}, \mathds{1} \rangle_{\eta} = \langle \nabla_{\eta}g, \mathds{1} \rangle_{\eta} = 0$, we have the following chain of equalities, which yields the result.
\begin{align*}
\int (\nabla_{\eta}g) f \ \textup{d}\eta
&= \int \nabla_{\eta}g (y) \left( f(1) + \int \mathds{1}_{(y, 1]}(x) f^{*}(x) \ \mathrm{d}\eta(x) \right) \ \mathrm{d}\eta(y)\\
&= \int f^{*}(x) \int \nabla_{\eta} g(y) \mathds{1}_{[0, x)}(y) \ \mathrm{d}\eta(y) \mathrm{d}\eta(x)
= \int f^{*}(x) g(x) \ \mathrm{d}\eta(x) \qedhere
\end{align*}
\end{proof}

\begin{cor}
For $f \in \textup{Dom}(\nabla_{\eta}^{*})$, we have that $\nabla_{\eta}^{*} f = f^{*}$, where $f^{*}$ is defined as in \eqref{eq:adjoint_domain}.
\end{cor}

The function $\nabla_{\eta}^{*} f$, when it exists, reflects local properties of $f$.  Indeed, we have
\begin{align}\label{eq:derivative_calculation*}
\nabla_{\eta}^{*} f(z_i) = \lim_{\varepsilon \searrow 0} \frac{\mathbf{f}(z_{i}) - \mathbf{f}(z_{i} - \varepsilon)}{\alpha_i}
\quad \text{and} \quad 
\nabla_{\eta}^{*}f(x) = f_{\Lambda}^{*}(x),
\end{align}
for $i \in \{ 1, \ldots, N \}$ and $x \in A_{i}^{\mathrm{o}} \cap \I$, where $f_{\Lambda}^{*} \in L^{2}_{\Lambda}$ is equal to the negative of the weak derivative of $f$ on $A_{i}^{\mathrm{o}}$, if $c_{i}=1$ and otherwise can be chosen arbitrarily on $A_{i}^{\mathrm{o}}$. 

\begin{prop}\label{thm:nabla_closed}
The operator $\nabla_{\eta}$ is densely defined, unbounded and closed.
\end{prop}

\begin{proof}
By \Cref{prop:D1_dense_in_L2}, the operator $\nabla_{\eta}$ is densely defined. Set $K \coloneqq \Gamma(\I)$ and define for $m \in \N$ and $x\in \I$ $g_m (x) = \sin (2 \pi m F_{\Gamma}(x)/K) \in \mathscr{D}_{\eta}^{1}$, in which case 
\begin{align*}
\nabla_{\eta} g_m (x) = \begin{cases}
(2 \pi m /K) \cos (2 \pi m  F_{\Gamma}(x)/K) & \text{if} \; x \in \I \setminus \{ z_1, \ldots, z_N\},\\
0 & \text{otherwise}.
\end{cases}
\end{align*}
This implies that $\nabla_{\eta}$ is an unbounded operator as $\lVert g_m \rVert_{\eta} \leq 1 + \sum_{i=1}^{N} \alpha_{i}$ and $\lVert \nabla_{\eta} g_m \rVert_{\eta} = \sqrt{2} \pi m$. \Cref{thm:Reed_Simon_VIII.3} gives that $\nabla_{\eta}^{**}$ is closed. To complete the proof, it suffices to show that $\textup{Dom}(\nabla_{\eta}) = \textup{Dom}(\nabla_{\eta}^{**})$.  However, this follows by an analogously argument to that given in the proof of \Cref{lem:d1mu*}; noting, by a similar proof to that of \Cref{prop:D1_dense_in_L2}, we have $\mathscr{D}_{\eta}^{1*}$ is dense in $L^2_{\eta}$ with respect to $\lVert \cdot \rVert_{\eta}$.
\end{proof}

The non-negative symmetric bilinear form $\E \colon \mathscr{D}_{\eta}^{1} \times \mathscr{D}_{\eta}^{1} \to \R$ defined by $\E(f,g) = \E_{\eta}(f,g) \coloneqq \langle \nabla_{\eta}f,\nabla_{\eta}g \rangle_{\eta}$ is called the \textsl{$\eta$-energy form}. In our next result, we show that $\E$ is a Dirichlet form. With this at hand, we may then define the $\eta$-Laplacian $\Delta_{\eta}$.

\begin{prop}\label{prop:dirichlet_form}
The $\eta$-energy form $\E$ is a Dirichlet form with domain $\mathscr{D}_{\eta}^{1}$.
\end{prop}

\begin{proof}
Using the properties of the inner product $\langle \cdot,\cdot \rangle_{\eta}$ and the operator $\nabla_{\eta}$ it follows that $\E$ is bilinear, symmetric and that $\E(u, u) \geq 0$, for all $u \in \mathscr{D}_{\eta}^{1}$. Moreover, this yields that $\mathscr{D}_{\eta}^{1}$ equipped with $\langle \cdot, \cdot \rangle_{\E} \coloneqq \langle \cdot, \cdot \rangle_{\eta} + \E(\cdot, \cdot)$ is an inner product space.  All that remains is to show that $\mathscr{D}_{\eta}^{1}$ is complete with respect to the norm induced by $\langle \cdot, \cdot \rangle_{\E}$ and that the Markov property holds.

If $(f_{n})_{n \in \N}$ is a Cauchy sequence in $( \mathscr{D}_{\eta}^{1}, \langle \cdot, \cdot \rangle_{\E} )$, then both $(f_{n})_{n \in \N}$ and $( \nabla_{\eta}f_{n} )_{n \in \N}$ are Cauchy-sequences in $L_{\eta}^{2}$. Hence, there exist $\widetilde{f_{0}}, \widetilde{f_{1}} \in L_{\eta}^{2}$ with $\lim_{n\to \infty} \lVert f_n - \widetilde{f_{0}} \rVert_{\eta} = 0$ and $\lim_{n\to \infty} \lVert \nabla_{\eta}f_n - \widetilde{f_{1}} \rVert_{\eta} = 0$.  \Cref{thm:nabla_closed} implies that $\nabla_{\eta} \widetilde{f_{0}} = \widetilde{f_{1}}$.  Thus, $ (f_{n})$ converges to $\widetilde{f_{0}} \in \mathscr{D}_{\eta}^{1}$ with respect to the norm induced by $\langle \cdot, \cdot \rangle_{\E}$.

For the Markov property it is sufficient to show, for $u \in \mathscr{D}_{\eta}^{1}$, that $u_{+} \in \mathscr{D}_{\eta}^{1}$ and $\vert \nabla_{\eta}u_{+}(x) \rvert \leq \vert \nabla_{\eta}u(x) \rvert$ for all $x \in \I$. Here $u_{+} \coloneqq \min(\max(u,0),1)$.  Define, for $x \in \I \setminus \{ z_{1}, \ldots, z_{N} \}$,
\begin{align*}
u_{+}^{\prime}(x) \coloneqq \begin{cases}
\nabla_{\eta} u(x) & \text{if} \; u(x) \in [0, 1],\\
0 & \text{otherwise}.
\end{cases}
\end{align*}
and for $i \in \{ 1, \ldots, N\}$ set $u_{+}^{\prime}(z_{i}) \coloneqq \lim_{\varepsilon \searrow 0} ( \mathbf{u}_{+}(x + \varepsilon) - \mathbf{u}_{+}(x) )/\alpha_{i}$. A direct calculation shows, for $x,y \in \R$ with $x<y$, that this function fulfils $\mathbf{u}_{+}(x) = \mathbf{u}_{+}(y) + \int \mathds{1}_{[y,x)} \mathbf{u}_{+}^{\prime} \ \mathrm{d}\eta * \delta_{\Z}$.  Hence, $u_{+} \in \mathscr{D}_{\eta}^{1}$ with $\nabla_{\eta} \mathbf{u}_{+} = \mathbf{u}_{+}^{\prime}$. By definition, $\lvert \nabla_{\eta} \mathbf{u}_{+}(x) \rvert \leq \lvert \nabla_{\eta} \mathbf{u}(x) \rvert$ for all $x \in \I \setminus \{ z_{1}, \ldots, z_{N} \}$. 
In the case that $x \in \{ z_{1}, \ldots, z_{N} \}$, we have
\begin{align*}
\lim_{\varepsilon \searrow 0} \lvert \mathbf{u}^{+}(x + \varepsilon) - \mathbf{u}^{+}(x) \rvert
\leq \lim_{\varepsilon \searrow 0} \lvert \mathbf{u}(x + \varepsilon) - \mathbf{u}(x) \rvert,
\end{align*}
and so, by \eqref{eq:derivative_calculation}, it follows that $\lvert \nabla_{\eta} \mathbf{u}_{+}(z_{i}) \rvert \leq \lvert \nabla_{\eta} \mathbf{u}(z_{i}) \rvert$ for $i \in \{ 1, \ldots, N\}$.
\end{proof}

We write $f \in \mathscr{D}_{\eta}^{2}$ if $f \in \mathscr{D}_{\eta}^{1}$ and if there exists $h \in L^2_{\eta}$ such that $\E(f,g) = - \langle h,g \rangle_{\eta}$, for all $g \in \mathscr{D}_{\eta}^{1}$. We call the operator $\Delta_{\eta} \colon \mathscr{D}_{2}^{\eta} \to L^2_{\eta}$ defined by $\Delta_{\eta} f \coloneqq h$ the \textsl{$\eta$-Laplacian}.  Notice, for an arbitrary $g \in \mathscr{D}_{\eta}^{1}$, that $\langle \nabla_{\eta}f,\nabla_{\eta}g \rangle_{\eta}  = - \langle \Delta_{\eta} f,g \rangle_{\eta}$, and thus $\Delta_{\eta} = -\nabla_{\eta}^{*} \circ \nabla_{\eta}$.

\begin{thm}\label{thm:Delta_sel_adjoint}
The operator $\Delta_{\eta}$ is densely defined on $L^2_{\eta}$, linear, self-adjoint and non-positive.
\end{thm}

\begin{proof}
That the operator is densely defined follows from the observation that the functions $f_{\varepsilon,n}, g_{\varepsilon,n}, h$ and $h_{n}$, as defined in \Cref{prop:D1_dense_in_L2}, lie in $\mathscr{D}_{\eta}^{2}$. Linearity follows from the linearity of $\nabla_{\eta}$ and the bilinearity of the inner product. The fact that $\Delta_{\eta}$ is self-adjoint is a consequence of \Cref{thm:Reed_Simon_VIII.3} and \Cref{thm:nabla_closed}.  Further, since $\langle \Delta_{\eta} f,f \rangle_{\eta} = - \langle \nabla_{\eta}f,\nabla_{\eta}f \rangle_{\eta} \leq 0$ we have the operator $\Delta_{\eta}$ is non-positive.
\end{proof}

We conclude this section with the following theorem, in which we show that $\Delta_{\eta}$ has compact resolvent. For this we use the following notation. We denote the closed unit ball in a normed space $(X, \lVert \cdot \rVert)$ by $B(X,\lVert \cdot \rVert)$ and for $i \in \{ 1, \ldots, N \}$, let $(W^{1,2}_{i},\lVert \cdot \rVert_{i})$ denote the Sobolev space $(W^{1,2}(A_{i}^{\textup{o}}), \lVert \cdot \rVert_{1,2})$.

\begin{thm}\label{thm:resolvent}
The operator $\Delta_{\eta}$ has compact resolvent.
\end{thm}

\begin{proof}
Let $\lambda$ denote a fixed element of the resolvent set . We show the embedding $\pi \colon \mathscr{D}_{\eta}^{1} \to L^{2}_{\eta}$ is compact, and that $(\lambda \operatorname{Id} - \Delta_{\eta})^{-1} \colon L^{2}_{\eta} \to  \mathscr{D}_{\eta}^{1}$ is continuous. This is sufficient to prove the result since the composition of a compact operator and a continuous operator is compact.

Let $(f_{n})_{n \in \N}$ be a sequence in $B(\mathscr{D}_{\eta}^{1}, \lVert \cdot \rVert_{\E})$.  To show $\pi$ is compact, it is sufficient to show that $(f_{n})_{n \in \N}$ has a convergent subsequence with respect to $\lVert \cdot \rVert_{\eta}$. It is known, for $i \in \{ 1, \ldots, N \}$, that the unit ball  $B(W^{1,2}_{i}, \lVert \cdot \rVert_{i})$ is compact in $L^{2}_{\Lambda}$. By \eqref{eq:derivative_calculation}, if $f \in B(\mathscr{D}_{\eta}^{1}, \lVert \cdot \rVert_{\E})$, then $f \rvert_{A_{i}^{\textup{o}}} \in B(W^{1,2}_{i}, \lVert \cdot \rVert_{i})$. This yields the existence of a subsequence $(n_{k})_{k \in \N}$ such that for $i \in \{ 1, \ldots, N \}$ with $c_{i} = 1,$ the sequence $(f_{n_{k}} \rvert_{A_{i}^{\textup{o}}})_{k \in \N}$ converges in $W^{1,2}_{i}$. Combined with the Bolzano-Weierstrass Theorem and the left continuity of elements in $\mathscr{D}_{\eta}^{1}$, this yields the existence of a subsequence $(n_{l})_{l \in \N}$, such that  $(f_{n_{l}})_{l \in \N}$ converges with respect to $\lVert \cdot \rVert_{\eta\rvert_{A_{i}}}$ for all $i \in \{ 1, \ldots, N \}$, and hence, with respect to $\lVert \cdot \rVert_{\eta}$.

To conclude the proof it is sufficient to find, for each $\lambda$ belonging to the resolvent set, a constant $C \in \R$ such that $\lVert (\Delta_{\eta} - \lambda \operatorname{Id})^{-1}f \rVert_{\E} \leq  C \lVert f \rVert_{\eta}$ for all $f \in L^{2}_{\eta}$. This is done in the following sequence of inequalities, in which we use that, since $(\Delta_{\eta} - \lambda \operatorname{Id})^{-1}$ is a bounded linear operator on $L^{2}_{\eta}$, there exists a $K \in \R$ with $\lVert (\Delta_{\eta} - \lambda \operatorname{Id})^{-1}f \rVert_{\eta} \leq  K \lVert f \rVert_{\eta}$.
\begin{align*}
\lVert (\Delta_{\eta} - \lambda \operatorname{Id})^{-1}f \rVert_{\E}^{2} 
&= \langle  (\Delta_{\eta} - \lambda \operatorname{Id})^{-1}f ,  (\Delta_{\eta} - \lambda \operatorname{Id})^{-1}f \rangle_{\eta} + \langle \nabla_{\eta} (\Delta_{\eta} - \lambda \operatorname{Id})^{-1}f , \nabla_{\eta}(\Delta_{\eta} - \lambda \operatorname{Id})^{-1}f  \rangle_{\eta} \\ 
&\leq  (1+ \lvert \lambda \rvert) \,\lVert (\Delta_{\eta} - \lambda \operatorname{Id})^{-1}f \rVert_{\eta}^{2} + \langle f , (\Delta_{\eta} - \lambda \operatorname{Id})^{-1}f  \rangle_{\eta} \\
&\leq \left(\left(1+\lvert \lambda \rvert \right) K^{2}+K\right)\lVert f \rVert_{\eta}^{2}
\qedhere
\end{align*}
\end{proof}

Since every eigenfunction of the resolvent operator is also an eigenfunction of $\Delta_{\eta}$, the spectral theorem for compact operators together with the fact that $\Delta_{\eta}$ is non-positive imply that the eigenfunctions of $\Delta_{\eta}$ form a basis of $L_{\eta}^{2}$ and that the eigenvalues of $\Delta_{\eta}$ are non-positive and form a countable unbounded monotonic sequence. Moreover, all eigenvalues of $\Delta_{\eta}$ have finite multiplicity.

\begin{proof}[Proof of \Cref{thm:main1}]
This is a direct consequence of \Cref{rmk:KSW16_2} and \Cref{thm:Delta_sel_adjoint,thm:resolvent}.
\end{proof}

\section{Eigenvalues and eigenfunctions of $\Delta_{\eta}$}\label{sec:eigenvalues_functions}

\subsection{Proof of \Cref{thm:main2} and general observations}\label{sec:3_1}

Let $N \in \N$ denote a positive integer and for $i \in \{ 1, \ldots, N\}$ let $\alpha_{i} > 0$ and $z_{i} \in \I$ with $0 < z_{1} < \cdots < z_{N} \leq 1$.  In this section we determine a systematic way to compute the eigenvalues and eigenfunctions of the $\eta$-Laplacian, where $\eta = \Lambda + \sum_{i=1}^{N} \alpha_{i} \delta_{z_{i}}$.  Without loss of generality, the assumption $z_{N}=1$ can be made, since we can obtain the eigenfunctions of measures not having this property by applying an appropriate translation argument. To be precise, set $\hat{z}_{i} \coloneqq z_{i}+1-z_{N}$ for $i \in \{ 1, \ldots, N\}$ and define the measure $\hat{\eta} = \Lambda + \sum_{i=1}^{N} \alpha_{i} \delta_{\hat{z}_{i}}$.  By construction, we have that $f$ is an eigenfunction of $\Delta_{\hat{\eta}}$ if and only if 
\begin{align*}
x \mapsto \begin{cases}
f(x - z_{N} + 1) & x \in (0, z_{N}],\\
f(x - z_{N}) & x \in (z_{N}, 1],
\end{cases}
\end{align*}
is an eigenfunction of $\Delta_{\eta}$. 

\begin{proof}[Proof of \Cref{thm:main2}]
Combining \eqref{eq:derivative_calculation}, \eqref{eq:derivative_calculation*} and the fact that elements of $\textup{Dom}(\nabla_{\eta})$ are left-continuous, together with Picard-Lindel\"of's Theorem, if $\Delta_{\eta} f = \lambda f$, for some $\lambda \in \R$ and $f \in \mathscr{D}_{\eta}^{2} $, then there exist $b, a_{1}, \ldots, a_{N}, \gamma_{1},\ldots, \gamma_{N} \in \R$ with $\gamma_{1} \in ( -\pi/2, \pi/2 ]$ and $\lambda = -b^2$, such that 
\begin{align}
f(x) =
\begin{cases}\label{eq:eigenfunction_form}
a_{1} \sin(b x + \gamma_{1}) & \text{if} \; x \in (0, z_{1}],\\
\hspace{3em} \vdots & \hspace{2.5em} \vdots \\
a_{N} \sin(b x + \gamma_{N}) & \text{if} \; x \in (z_{N-1}, 1].
\end{cases}
\end{align}
Since $\mathbf{\nabla_{\eta} f}$ is right-continuous and $\mathbf{\Delta_{\eta} f}$ is left-continuous, by \eqref{eq:derivative_calculation} and \eqref{eq:derivative_calculation*}, we have $f$ is an eigenfunction of $\Delta_{\eta}$ if and only if $b, a_{1}, \ldots, a_{N}, \gamma_{1},\ldots, \gamma_{N}$ satisfy the following system of equations:
\begin{align}\label{eq:system_general}
\begin{aligned}
&\alpha_{j} b a_{j+1} \cos(b z_{j} +\gamma_{j+1}) = a_{j+1} \sin(b z_{j} +\gamma_{j+1}) - a_{j} \sin(b z_{j} +\gamma_{j}), \\
&\alpha_{j} b^2 a_{j} \sin(b z_{j} +\gamma_{j}) =  a_{j} b \cos(b z_{j} +\gamma_{j}) - a_{j+1} b \cos(b z_{j} +\gamma_{j+1}),
\end{aligned}
\end{align}
for $j \in \{1,\ldots,N-1\}$, and
\begin{align}\label{eq:system_general_N}
\begin{aligned}
&\alpha_{N} b a_{1} \cos(\gamma_{1}) = a_{1} \sin(\gamma_{1}) - a_{N} \sin(b +\gamma_{N}), \\
&\alpha_{N} b^2 a_{N} \sin(b +\gamma_{N}) = a_{N} b \cos(b +\gamma_{N}) - a_{1} b \cos(\gamma_{1}).
\end{aligned}
\end{align}
This concludes the proof.
\end{proof}

Let $f$ be an eigenfunction of $\Delta_{\eta}$ with eigenvalue $\lambda$, and let $b, a_{1}, \dots, a_{N}$ be as in \eqref{eq:eigenfunction_form}. By the fact that $\lambda = -b^{2}$, if $b=0$, then $f$ is a step function, and so, \eqref{eq:system_general} and \eqref{eq:system_general_N} imply that $f$ is a constant function.  If $a_{i} = 0$ for some $i \in \{ 1, \ldots, N-1\}$ and if $b \neq 0$, then \eqref{eq:system_general} yields that $a_{j+1}\cos(bz_{j}+\gamma_{j+1}) = a_{j+1}\sin(bz_{j}+\gamma_{j+1}) = 0$, and hence, $a_{j+1} = 0$. An analogue results holds when one assumes that $a_{N} = 0$.  This implies, if $a_{i} = 0$ for some $i \in \{ 1, \ldots, N \}$, then $f \equiv 0$.

\begin{cor}\label{cor:constant_ef_a_nonzero}
Every constant function is an eigenfunction with corresponding eigenvalue equal to zero. Moreover, the eigenspace $E_{0} \coloneqq \{ f \in \mathscr{D}_{\eta}^{2} \colon \Delta_{\eta} f \equiv 0 \}$ is one-dimensional.  Further, if $f \not\in E_{0}$ is an eigenfunction of $\Delta_{\eta}$, then $b \neq 0$ and $a_{i} \neq 0$ for all $i \in \{1, \ldots, N\}$, where $b, a_{1}, \dots, a_{N}$ are as in \eqref{eq:eigenfunction_form}.
\end{cor}

If the atoms are equally distributed, namely $z_{i}-z_{i-1} = 1/N$ for $i \in \{ 1, \ldots, N \}$, two properties which follow directly from \eqref{eq:system_general} and \eqref{eq:system_general_N} are the following. If $f$ is an eigenfunction of $\Delta_{\eta}$, then, for $r \in \{2, \ldots, N \}$,
\begin{align}
f_{r}(x) &=
\begin{cases}
a_{N-r+2} \sin(b (x +(N-r+1)/N) + \gamma_{N-r+2}) & \text{if} \; x \in (0, 1/N],\\
\hspace{5em}\vdots &\hspace{3.5em}\vdots \\
a_{N} \sin(b (x+(N-r+1)/N) + \gamma_{N}) & \text{if} \; x \in ((r-2)/N, (r-1)/N],\\
a_{1} \sin(b (x-(r-1)/N) + \gamma_{1}) & \text{if} \; x \in ((r-1)/N, r/N],\\
\hspace{5em}\vdots &\hspace{3.5em}\vdots \\
a_{N-r+1} \sin(b (x-(r-1)/N) + \gamma_{N-r+1}) & \text{if} \; x \in ((N-1)/N, 1],
\end{cases}\label{eq:rotation}
\end{align}
is an eigenfunction of $\Delta_{\eta_{r}}$, where $\eta_{r} \coloneqq \Lambda + \sum_{i=1}^{r-1} \alpha_{i+N-r+1} \delta_{z_{i}} + \sum_{i=r}^{N} \alpha_{i-r+1} \delta_{z_{i}}$.  Note, if $\alpha_{1} = \ldots = \alpha_{N}$, then $\eta_{r} = \eta$. 

\begin{cor}\label{cor:concatenation}
Let $N=p k$ with $p,k \in \N$ and suppose that $\alpha_{i+p} = \alpha_{i}$ for all $i \in \{ 1, \ldots, N-p\}$. Set 
\begin{align*}
\eta^{(p)} \coloneqq \Lambda + \sum_{i=1}^{p} \alpha_{i} \delta_{i/p}
\quad \text{and} \quad 
\eta^{(N)} \coloneqq \Lambda + \sum_{i=1}^{N} \alpha_{i}/k \delta_{i/N}.
\end{align*}
If $f^{(p)}$ is an eigenfunction of $\Delta_{\eta^{(p)}}$ with eigenvalue $\lambda^{(p)}$, then $f^{(N)}$ is an eigenfunction of $\Delta_{\eta^{(N)}}$ with eigenvalue $\lambda^{(n)} \coloneqq k^{2}\lambda^{(p)}$, where $f^{(N)}(x) \coloneqq \mathbf{f}^{(p)} (kx)$ for $x \in \I$.
\end{cor}

\subsection{$\mathbf{N = 1}$}\label{sec:N=1}

Here we consider the case $N=1$, namely when $\eta = \Lambda +\alpha \delta_{z}$, for some $\alpha > 0$ and $z \in \I$. As in \Cref{sec:3_1}, without loss of generality we may assume that $z = 1$. The main results of this section are \Cref{thm:eigenvalues_and_functions} and \Cref{cor:evcf_1atom}, in which we explicitly compute the eigenvalues and eigenfunctions of $\Delta_{\eta}$. For the proofs of these results we will require the following preliminaries. For $\beta > 0$ and $k \in \Z$, let $c_{k}^{\pm} = c_{k}^{\pm}(\beta)$ denote the unique solution in the interval $(-\pi/2, \pi/2)$ to the equation
\begin{align}\label{eq:def_ck} 
\tan(c_{k}^{\pm}) = -2c_{k}^{\pm} \beta \pm \beta \pi/2 + 2 \pi \beta k \pm 1
\end{align}
and set 
\begin{align}\label{def:xi_k}
\xi_{k}^{\pm} = \xi_{k}^{\pm}(\beta) \coloneqq  (\tan(c_{k}^{\pm}) \mp 1)/\beta.
\end{align}
Analogously, for $\beta < 0$ we denote by $C_{k}^{\pm} = C_{k}^{\pm}(\beta)$ the set of solutions to \eqref{eq:def_ck}. Note, the cardinality $\lvert C_{k}^{\pm} \rvert$ of $C_{k}^{\pm}$ is less than or equal to three. We denote the elements of $C_{k}^{\pm}$ by $c_{k,i}^{\pm} = c_{k,i}^{\pm}(\beta)$, where $i \in \{ 1,\ldots, \lvert C_{k}^{\pm}\rvert\}$. For every $c_{k,i}^{\pm}$ define the values $\xi_{k,i}^{\pm} = \xi_{k,i}^{\pm}(\beta)$ similar to \eqref{def:xi_k} and denote by $\Xi_{k}^{\pm} = \Xi_{k}^{\pm}(\beta)$ the set $\{ \xi_{k,i}^{\pm} \colon i \in \{ 1,\ldots, \lvert C_{k}^{\pm}\rvert\} \}$.

\begin{lem}\label{lem:sol_equations}
Let $\beta \in \R \setminus \{ 0 \}$. The pair $(\xi,c) \in \R \setminus \{0\} \times \R$ is a solution to the system of equations
\begin{align}\label{eq:one_atom_system_general}
\begin{aligned}
&\beta \xi \cos(c) = \sin(c) - \sin(\xi + c)\\
&\beta \xi^{2} \sin(\xi + c) = \xi \cos(\xi + c) - \xi \cos(c)
\end{aligned}
\end{align}
if and only if
\begin{align*}
(\xi,c) \in 
\begin{cases}
\{ (\xi_{k}^{\pm},c_{k}^{\pm}) \colon k \in \Z \} &\text{if} \; \beta >0, \\
\{ (\xi_{k,i}^{\pm}, c_{k,i}^{\pm}) \colon k \in \Z \; \text{and} \;  i \in \{ 1,\ldots, \lvert C_{k}^{\pm}\rvert\} \} &\text{if} \; \beta<0. 
\end{cases}
\end{align*}
The system of equations given in \eqref{eq:one_atom_system_general} is also solved by $(0,c)$ for all $c \in \R$.  Further, if $c = \pi /2 + \pi k$, for some $k \in \Z$, then the only solution to \eqref{eq:one_atom_system_general} is given when $\xi = 0$.
\end{lem}

\begin{proof}
The backwards implication follows by substituting the given values for $\xi$ and $c$ directly into \eqref{eq:one_atom_system_general}.  We now show the forwards implication.  Substituting the first equation of \eqref{eq:one_atom_system_general} into the second equation of \eqref{eq:one_atom_system_general}, and using the identity $\cos^{2}(\arcsin(x)) = 1 -x^{2}$, for $x \in [-1, 1]$, we obtain
\begin{align*}
(\beta \xi  \sin(c) - \beta^{2} \xi^{2} \cos(c) + \cos(c))^{2} = 1 - (\sin(c) - \beta \xi \cos(c))^{2}.
\end{align*}
From this it follows that $\cos(c) \neq 0$ and hence,
\begin{align*}
\beta^{2} \xi^{2} - 2 \tan(c) \beta \xi + (\tan^{2}(c) - 1) = 0.
\end{align*}
Thus, we have that either $\xi = 0$ and $c = \pm \pi/4 + \pi k$, for some $k \in \Z$, or that $\beta \xi = \tan(c) \pm 1$.  In the case $\xi \neq 0$, substituting this value into the first equation of \eqref{eq:one_atom_system_general} yields $( \tan(c) \pm 1 ) \cos(c) = \sin(c) - \sin ( (\tan(c) \pm 1)/\beta + c )$, or equivalently,
\begin{align*}
\cos(c) = \sin \left( \frac{\mp \tan(c) - 1}{\beta} \mp c \right).
\end{align*}
This leads to the following four cases:
\begin{enumerate}
\item\label{case1} $\beta \xi = \tan(c) + 1$ and $\tan(c) = -2 c \beta - \beta\pi/2 + 2 \pi \beta k - 1$, for $k \in \Z$, 
\item\label{case2} $\beta \xi = \tan(c) + 1$ and $\tan(c) = - \beta\pi/2 + 2 \pi \beta k - 1$, for $k \in \Z$,
\item\label{case3} $\beta \xi = \tan(c) - 1$ and $\tan(c) = -2 c \beta + \beta\pi/2 + 2 \pi \beta k + 1$, for $k \in \Z$, or
\item\label{case4} $\beta \xi = \tan(c) - 1$ and $\tan(c) = \beta\pi/2 + 2 \pi \beta k + 1$, for $k \in \Z$.
\end{enumerate}
By substituting these values into \eqref{eq:one_atom_system_general}, one sees that Cases \ref{case1} and \ref{case3} yield solutions to \eqref{eq:one_atom_system_general}. Cases \ref{case2} and \ref{case4} do not yield solutions, except when $c = 0$, but this is the same solution given by Cases \ref{case1} and \ref{case3} when $c = 0$. 

The last statement follows by substituting the given values for $\xi$ and $c$ directly into \eqref{eq:one_atom_system_general}.
\end{proof}

By \eqref{eq:eigenfunction_form}, if $\Delta_{\eta} f = \lambda f$, for some $\lambda \in \R$, then there exist $a, b \in \R$ and $\gamma \in (-\pi/2, \pi/2]$ such that $\lambda = -b^2$ and $f(x) = a \sin(b x + \gamma)$ for $x \in \I$. By linearity, without loss of generality we may assume $a = 1$. Further, \eqref{eq:system_general} and \eqref{eq:system_general_N} imply that $f$ is an eigenfunction of $\Delta_{\eta}$ if and only if $b$ and $\gamma$ satisfy the following system of equations.
\begin{align}\label{eq:system1}
\begin{aligned}
&\alpha b \cos(\gamma) = \sin(\gamma) - \sin(b + \gamma)\\
&\alpha b^{2} \sin(b + \gamma) = b \cos(b + \gamma) - b \cos(\gamma)
\end{aligned}
\end{align}
Thus, if $f$ is non-constant, then $\gamma \neq \pi/2$. Suppose, by way of contradiction, that $\gamma = \pi/2$. In this case \eqref{eq:system1} implies that $0 = 1 - \sin(b + \pi/2)$ and $\alpha b^{2} \sin(b + \pi/2) = b \cos(b + \pi/2)$.  The first yields that $b = 2 \pi n$, for some $n \in \Z$.  Substituting this value for $b$ into the second equation yields $\alpha (2 \pi n)^{2} = 0$, and so $n = 0$. Hence, $b = 0$, in which case $f = \mathds{1}$.

For $k \in \Z$, let $\gamma = \gamma^{(k,1)}(\alpha)$ denote the unique solution in the interval $(-\pi/2, \pi/2)$ to $\tan(\gamma) = -2\gamma \alpha + \alpha \pi/2 + 2 \pi \alpha k + 1$, and set $b^{(k,1)}(\alpha) \coloneqq  -2\gamma^{(k,1)}(\alpha) + \pi/2 +2 \pi k$.  As we will shortly see, $b^{(k,1)}$ and $\gamma^{(k,1)}$ completely determine the eigenfunctions and eigenvalues of $\Delta_{\eta}$.  We have introduced the extra index $1$ to indicate that they give rise to solutions to the eigenvalue problem when $\eta$ has a single atom; this will become important in the subsequent section where we consider measures with two atoms.

Notice, if $k = 0$, then $\gamma^{(k,1)} = \pi/4$ and $b^{(k,1)} = 0$; if $b^{(k,1)} = 0$, then $\gamma^{(k,1)} = \pi/4$ and $k = 0$; if $\gamma^{(k,1)} = \pi/4$, then $k = 0$ and $b^{(k,1)} = 0$.

\begin{thm}\label{thm:eigenvalues_and_functions}
The eigenvalues of $\Delta_{\eta}$ are $\lambda^{(k,1)} = - (b^{(k,1)}(\alpha))^{2}$ for $k \in \Z$, with corresponding eigenfunctions $f^{(k,1)}(x) \coloneqq \sin(b^{(k,1)}(\alpha) x + \gamma^{(k,1)}(\alpha))$.  Further, each eigenvalue has multiplicity one.
\end{thm}

\begin{proof}
This is a direct consequence of \Cref{lem:sol_equations} with $\beta =\alpha$, $c_{k}^{+} =  \gamma^{(k,1)}$ and $\xi_{k}^{+} = b^{(k,1)}$ and the observation that $-c_{-k}^{-} = c_{k}^{+}$ and $-\xi_{-k}^{-} = \xi_{k}^{+}$.
\end{proof}

Note, the only eigenfunction $f^{(k,1)}$ with $\mathbf{f}^{(k,1)}$ continuous is the constant function $f^{(0,1)}$.  Indeed, if there exists $k \in \Z \setminus \{ 0 \}$ with $\mathbf{f}^{(k,1)}$ continuous, then $\sin( b^{(k,1)} + \gamma^{(k,1)}) = \sin( \gamma^{(k,1)})$, and so, by \eqref{eq:system1}, we would have $\cos(\gamma^{(k,1)}) = 0$ as $b^{(k,1)} \neq 0$, contradicting the fact that $\gamma = \gamma^{(k,1)}(\alpha) \in (-\pi/2, \pi/2)$. Further, for $k_{1},k_{2} \in \Z$ with $k_{1} \neq k_{2}$, by definition  $\gamma^{(k_{1},1)} \neq \gamma^{(k_{2},1)}$ and, since $\gamma^{(k,1)} \in (-\pi/2,\pi/2)$ for $k \in \Z \setminus \{ 0 \}$, we have $(b^{(k_{1},1)})^{2} \neq (b^{(k_{2},1)})^{2}$. Hence, all eigenvalues have multiplicity one.

\begin{cor}\label{cor:evcf_1atom}
Letting $N_{\eta} \colon \R^{+} \to \R$ denote the eigenvalue counting function of $-\Delta_{\eta}$, we have $\displaystyle \lim_{x \to \infty} \frac{\pi N_{\eta}(x)}{\sqrt{x}} = 1$.
\end{cor}

In contrast to the case when $\eta$ is atomless, the eigenvalues of $\Delta_{\eta}$ do not occur in pairs. Indeed, let $\lambda_{k}$ denote the $k$-th largest eigenvalue of $\Delta_{\eta}$, then
\begin{enumerate}
\item $\displaystyle \lim_{k \to \infty} \lambda_{k} / (k \pi +\pi/2)^2 = -1$,\\
\item $\displaystyle \lim_{k \to +\infty} \gamma^{(k,1)} = - \lim_{k \to -\infty} \gamma^{(k,1)} = \pi/2$, and\\
\item $\displaystyle \lim_{k \to +\infty} b^{(k,1)}/(2\pi k) = \lim_{k \to -\infty} b^{(k,1)}/(2\pi k) = 1$.
\end{enumerate}

\begin{ex}
For $\eta = \Lambda + \pi^{-1}\delta_{1}$ we have  $\lambda^{(1,1)} \approx -29.3$, $\lambda^{(2,1)} \approx -130.4$ and $\lambda^{(3,1)} \approx -309.1$; see \Cref{fig:ef_1_atom} for a graphical representation of $f^{(1,1)}$,  $f^{(2,2)}$ and  $f^{(3,1)}$.
\end{ex}

\begin{figure}[t]
\centering
\resizebox{0.875\linewidth}{!}{
\begin{subfigure}[t]{5.5cm}
\begin{tikzpicture}
%axis
  \draw[->] (-0.1,0) -- (4.5,0); 
  \draw[->] (0,-2.75) -- (0,2.75);
%atoms  
  \draw[thick,smooth] (4,-0.05) -- (4,0.05) node[below] {$1$};   
%function values  
    \draw[thick,smooth] (-0.05,2) -- (0.05,2) node[left] {$1$};  
    \draw[thick,smooth] (-0.05,-2) -- (0.05,-2) node[left] {$-1$}; 
%function  
  \draw[thick,domain=0:4,smooth, variable=\a] plot (\a,{2*sin ((5.416)*(\a/4 r) + 1.219 r )});
%circles  
  \filldraw (4,0.689) circle (1.4pt);
  \draw (0,1.878) circle (1.4pt);
%dashed lines 
  \draw[dashed,smooth] (4,0.1) -- (4,0.689);  
   \end{tikzpicture}
   \subcaption{Graph of $f^{(1,1)}$}
   \end{subfigure}
\qquad
\begin{subfigure}[t]{5.5cm}
\begin{tikzpicture}
%axis
  \draw[->] (-0.1,0) -- (4.5,0); 
  \draw[->] (0,-2.75) -- (0,2.75);
%atoms  
  \draw[thick,smooth] (4,-0.05) -- (4,0.05) node[below] {$1$};   
%function values  
    \draw[thick,smooth] (-0.05,2) -- (0.05,2) node[left] {$1$};  
    \draw[thick,smooth] (-0.05,-2) -- (0.05,-2) node[left] {$-1$}; 
%function  
  \draw[thick,domain=0:4,smooth, variable=\a] plot (\a,{2*sin ((11.421)*(\a/4 r) + 1.358 r )});
%circles  
  \filldraw (4,0.422) circle (1.4pt);
  \draw (0,1.955) circle (1.4pt);
%dashed lines 
  \draw[dashed,smooth] (4,0.1) -- (4,0.422);  
   \end{tikzpicture}
   \subcaption{Graph of $f^{(2,1)}$}
   \end{subfigure}
\qquad
\begin{subfigure}[t]{5.5cm}
\begin{tikzpicture}
%axis
  \draw[->] (-0.1,0) -- (4.5,0); 
  \draw[->] (0,-2.75) -- (0,2.75);
%atoms  
  \draw[thick,smooth] (4,-0.05) -- (4,0.05) node[below] {$1$};   
%function values  
    \draw[thick,smooth] (-0.05,2) -- (0.05,2) node[left] {$1$};  
    \draw[thick,smooth] (-0.05,-2) -- (0.05,-2) node[left] {$-1$}; 
%function  
  \draw[thick,domain=0:4,smooth, variable=\a] plot (\a,{2*sin ((17.580)*(\a/4 r) + 1.420 r )});
%circles  
  \filldraw (4,0.3) circle (1.4pt);
  \draw (0,1.977) circle (1.4pt);
%dashed lines 
  \draw[dashed,smooth] (4,0.1) -- (4,0.30);  
   \end{tikzpicture}
   \subcaption{Graph of $f^{(3,1)}$}
   \end{subfigure}}
\caption{Graphs of the eigenfunctions $f^{(k,1)}$ of $\Delta_{\eta}$ for $k \in \{1,2,3\}$, where $\eta = \Lambda + \pi^{-1} \delta_{1}$.}\label{fig:ef_1_atom}
\end{figure}
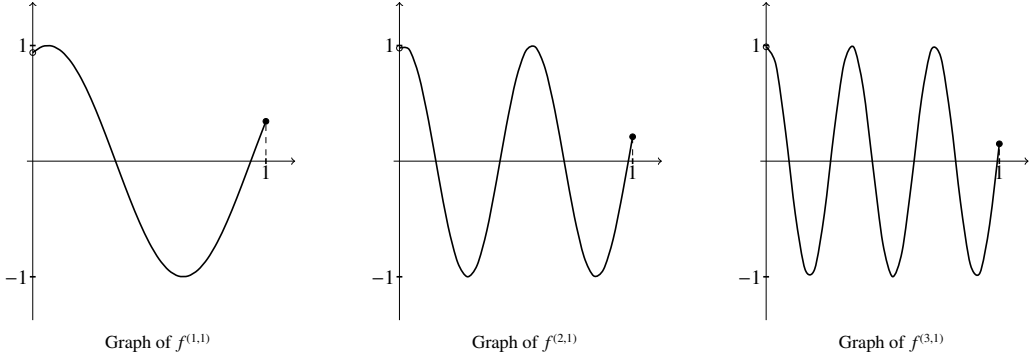

\subsection{$\mathbf{N = 2}$: Uniformly distributed Dirac point masses}\label{sec:3-2}

Let $\alpha$ denote a positive real number, let $z_{1}, z_{2} \in (0, 1]$ be such that $z_{2} - z_{1} = 1/2$ and let $\eta = \Lambda + \sum_{i = 1}^{2} \alpha \delta_{i/2}$.  As in \Cref{sec:N=1}, without loss of generality, we may assume that $z_{2} = 1$, and hence that $z_{1} = 1/2$.  The main results of this section are \Cref{thm:two_atom_case_thm} and \Cref{cor:evcf_2atom}, in which we determine the eigenvalues and eigenfunctions of $\Delta_{\mu}$.

By \eqref{eq:eigenfunction_form}, if $\Delta_{\eta} f = \lambda f$, for some $\lambda \in \R$, then there exist $b, a_{1}, a_{2}, \gamma_{1}, \gamma_{2} \in \R$ with $\gamma_{1} \in ( -\pi/2, \pi/2 ]$ such that
\begin{align*}
f(x) &=
\begin{cases}
a_{1} \sin(b x + \gamma_{1}) & \text{if} \; x \in (0, 1/2],\\
a_{2} \sin(b x + \gamma_{2}) & \text{if} \; x \in (1/2, 1].
\end{cases}
\end{align*}
\Cref{cor:constant_ef_a_nonzero} implies the constant function $\mathds{1}$ is an eigenfunction of $\Delta_{\eta}$ and the eigenspace $E_{0} = \{ f \in \mathscr{D}_{\eta}^{2} \colon \Delta_{\eta} f \equiv 0 \}$ is one-dimensional. In other words, if $b = 0$, then $f \in E_{0}$.

With this at hand we may assume that $b \neq 0$.  From the system of equations given in \eqref{eq:system_general} and \eqref{eq:system_general_N} it follows that if $f$ is an eigenfunction, then $a_{1}, a_{2}, b, \gamma_{1}, \gamma_{2}$ fulfil the following equations.  
\begin{align}\label{system_2}
\begin{aligned}
&\alpha b a_{1} \cos(\gamma_{1}) = a_{1} \sin(\gamma_{1}) - a_{2} \sin(b + \gamma_{2})\\
&\alpha b a_{2} \sin(b + \gamma_{2}) = a_{2} \cos(b + \gamma_{2}) - a_{1}\cos(\gamma_{1})\\
&\alpha b a_{2} \cos(b/2 + \gamma_{2}) = a_{2} \sin(b/2 + \gamma_{2}) - a_{1} \sin(b/2 + \gamma_{1})\\
&\alpha b a_{1} \sin(b/2 + \gamma_{1}) = a_{1} \cos(b/2 + \gamma_{1}) - a_{2}\cos(b/2 + \gamma_{2})
\end{aligned}
\end{align}
As discussed directly above \Cref{cor:constant_ef_a_nonzero}, we have $a_{1}, a_{2} \neq 0$, since otherwise \eqref{system_2} yields $f \equiv 0$.

By \eqref{eq:rotation}, we have that
\begin{align}
f_{2}(x) &= \begin{cases}
a_{2} \sin(b(x + 1/2) + \gamma_{2}) & \text{if} \; x \in (0, 1/2],\\
a_{1}\sin(b(x - 1/2) + \gamma_{1}) & \text{if} \; x \in (1/2, 1],
\end{cases}\label{eq:rotation_eigenfunction}
\end{align}
is also an eigenfunction of $\Delta_{\eta}$. Hence, without loss of generality, we may assume $a_{1} = 1$ and $\lvert a_{2} \rvert \leq 1$. Our aim is to find all tuples $(b,a_{2},\gamma_{1},\gamma_{2}) \in \R^{4}$ such that $f$ is a non-constant eigenfunction. We start with the special cases that $a_{2}=1$ and $b=\pm1/\alpha$. In the second step we discuss the case $a_{2}=1$ and $b \neq \pm1/\alpha$. Noting that $(b,a_{2},\gamma_{1},\gamma_{2})$ leads to an eigenfunction if and only if $(b,-a_{2},\gamma_{1},\gamma_{2} + \pi)$ does, solves the case $a_{2} = -1$. We then show that $f$ is not an eigenfunction if $\lvert a_{2} \rvert < 1$.

Suppose that $a_{2}=1$ and $b = -1/\alpha$. The first two equations in \eqref{system_2} implies that $\sin(\gamma_{1}) = -\cos(b+\gamma_{2})$. This yields that $b+\gamma_{2} = \gamma_{1} +\pi/2 +2 \pi k$ for some $k \in \Z$. Substituting this into the first equation of \eqref{system_2} implies that $\gamma_{1}=0$ and hence, $\gamma_{2} = -b + \pi/2 + 2 \pi k$ for some $k \in \Z$.  Subsequently, the third and fourth equations of \eqref{system_2} yields that $f$ is an eigenfunction if and only if $\alpha = 1/(\pi + 2 \pi m)$ for some $m\in \Z$, in which case $\gamma_{2} \bmod 2\pi =3 \pi/2$.

One can show in a similar manner, that if $a_{2}=1$ and $b = 1/\alpha$, then $f$ is an eigenfunction of $\Delta_{\eta}$ if and only if $\alpha = 1/(2\arctan(1/2)-2\arctan(2) + 2 \pi m)$ for some $m \in \Z$; in which case 
\begin{align*}
\gamma_{1}=\arctan(2)
\quad \text{and} \quad
\gamma_{2} = \arctan(2) -2\arctan(1/2) + \pi/2 \bmod 2\pi.
\end{align*}

\begin{cor}\label{cor:ef_special_cases}
If there exists $m\in \N_{0}$ with $\alpha = \alpha^{\prime} \coloneqq 1/(\pi + 2 \pi m)$ or $\alpha = \alpha^{\prime\prime} \coloneqq 1/(2\arctan(1/2)-2\arctan(2) + 2 \pi m)$, then $\lambda = -1/\alpha^{2}$ is an eigenvalue of $\Delta_{\eta}$ with multiplicity one. The corresponding eigenfunction is 
\begin{align*}
f(x) &=
\begin{cases}
\sin(-(1/\alpha) x) & \text{if} \; x \in (0, 1/2] \, \text{and} \, \alpha= \alpha^{\prime},\\
\sin(-(1/\alpha) x + 3\pi/2) & \text{if} \; x \in (1/2, 1] \, \text{and} \, \alpha= \alpha^{\prime},\\
\sin((1/\alpha) x + \arctan(2)) & \text{if} \; x \in (0, 1/2]  \, \text{and} \, \alpha=\alpha^{\prime\prime},\\
\sin((1/\alpha) x + \arctan(2) -2\arctan(1/2) + \pi/2) & \text{if} \; x \in (1/2, 1]  \, \text{and} \, \alpha=\alpha^{\prime\prime}.
\end{cases}
\end{align*}
\end{cor}

We now consider the case $a_{2} = 1$ and $b \in \R \setminus \{0, \pm1/\alpha \}$. From \eqref{system_2} it follows that $\cos(\gamma_{1}) \neq 0$ and $\cos(b/2 + \gamma_{2}) \neq 0$. This implies $f$ is discontinuous at the atoms and hence that $b+\gamma_{2}-\gamma_{1} \neq 0$ and $\gamma_{1}-\gamma_{2} \neq 0$; thus $\gamma_{1} \neq \pi/2$. Define $g \colon \I \to \R$ by $g(x) \coloneqq \sin ((b+\gamma_{2}-\gamma_{1})x+\gamma_{1})$ and set $\beta_{1} \coloneqq \alpha b / (b+\gamma_{2}-\gamma_{1})$. Setting $\beta = \beta_{1}$, $c=\gamma_{1}$ and $\xi = b+\gamma_{2}-\gamma_{1}$, the equalities in \eqref{system_2} imply those of \eqref{eq:one_atom_system_general}, and so there exists a $k \in \Z$ with $b+\gamma_{2} = -\gamma_{1} \pm \pi/2 + 2 \pi k$.  If $b+\gamma_{2} = -\gamma_{1} + \pi/2 + 2 \pi k$, then $\sin(b+\gamma_{2})= \cos(\gamma_{1})$ and $\sin(b/2+\gamma_{1})= \cos(b/2+\gamma_{2})$. Combining this with \eqref{system_2} yields
\begin{align}\label{eq:help3}
\begin{aligned}
\tan(\gamma_{1}) &= 1 + \alpha b, &\qquad
\tan(b+\gamma_{2}) &= 1/(1+\alpha b), &\\
\tan(b/2+\gamma_{2}) &= 1 + \alpha b, &\qquad
\tan(b/2+\gamma_{1}) &= 1/(1+\alpha b).&
\end{aligned}
\end{align}
For $k \in \Z$, let $\gamma_{1}^{(k,2)}$ denote the unique solution of $\tan (\gamma_{1}) = 1 - 4 \alpha \gamma_{1} + \alpha \pi +2 \pi \alpha k$.  Substituting the first equation of \eqref{eq:help3} into the last, we observe that $\gamma_{1} = \gamma_{1}^{(k,2)}$, $b=b^{(k,2)} = -4 \gamma_{1}^{(k,2)} + \pi + 2 \pi k$ and $\gamma_{2} = \gamma_{2}^{(k,2)} = -b^{(k,2)} -\gamma_{1}^{(k,2)} +\pi/2$, for some $k \in \Z$.  Observe that $\gamma_{1}^{(0,2)} = \pi/4$, and hence that $b^{(0,2)} = 0$, contradicting our initial assumption that $f$ is a non-constant eigenfunction. 

Similar to the case when $N=1$, if $b+\gamma_{2} = -\gamma_{1} - \pi/2 + 2 \pi k$, then analogue calculations yield the same eigenfunctions, namely that $b=-b^{(k,2)}$, $\gamma_{1} = -\gamma_{1}^{(k,2)}$ and $\gamma_{2} = -\gamma_{2}^{(k,2)}$

A direct calculation shows, for $x \in (0,1/2]$, that $\sin(b^{(k,2)} (x +1/2) + \gamma_{2}^{(k,2)}) = \sin(b^{(k,2)}x + \gamma_{1}^{(k,2)} - \pi k)$. This means that if $k$ is even, the eigenfunctions are periodic with period $1/2$, namely $f = f_{2}$ where $f_{2}$ is defined as in \eqref{eq:rotation_eigenfunction}. This means, as discussed in \Cref{cor:concatenation}, that these are concatenations of rescaled solutions for a measure with one atom, namely if $k=2m$, then $\gamma_{1}^{(k,2)}(\alpha) = \gamma^{(m,1)}(2 \alpha)$ and $b^{(k,2)}(\alpha) = 2 b^{(m,1)} (2 \alpha)$. On the other hand if $k$ is odd, then $f = -f_{2}$. Summarising, we have, for $k \in \Z \setminus \{ 0 \}$, that
\begin{align*}
f^{(k,2)}(x) =
\begin{cases}
\sin(b^{(k,2)} x +\gamma_{1}^{(k,2)}) & \text{if} \; x \in (0, 1/2],\\
\sin(b^{(k,2)}x  +\gamma_{2}^{(k,2)}) & \text{if} \; x \in (1/2, 1]
\end{cases} 
\end{align*}
is an eigenfunction of $\Delta_{\eta}$ with corresponding eigenvalue $\lambda^{(k,2)} = -(b^{(k,2)})^2$, and $f^{(0,2)} \coloneqq \mathds{1}$ is an eigenfunction of $\Delta_{\eta}$ with corresponding eigenvalue $\lambda^{(0,2)} \coloneqq 0$. 

If $\alpha$ is of the form discussed in \Cref{cor:ef_special_cases}, then the eigenfunction $f$ given there belongs to $\{ f^{(k,2)} \colon k\in \Z\}$.  Indeed, if there exists $m \in \N_{0}$ with $\alpha = 1/(\pi + 2 \pi m)$, then $f =  f^{(-m-1,2)}$, and if there exists $m \in \N$ with $\alpha = 1/(2\arctan(1/2)-2\arctan(2) + 2 \pi m)$, then $f =  f^{(m,2)}$.

To conclude, by way of contradiction, we show that there does not exist any other eigenfunction of $\Delta_{\eta}$ other than those discussed above. To this end, assume that $h$ is an eigenfunction of $\Delta_{\eta}$ in the orthogonal complement of $\operatorname{span} \{ f^{(k,2)} \colon k \in \Z \}$.  By the discussion following \eqref{eq:rotation}, we have
\begin{align*}
h(x) =
\begin{cases}
\sin(b x +\gamma_{1}) & \text{if} \; x \in (0, 1/2],\\
a_{2} \sin(b x  +\gamma_{2}) & \text{if} \; x \in (1/2, 1],
\end{cases} 
\end{align*}
for some $b \in \R$ and $a_{2}, \gamma_{1}, \gamma_{2} \in \R$ with $\lvert a_{2} \rvert < 1$. 
Letting $\psi(x) = x/2$ for all $x \in \R$, by definition and \Cref{thm:resolvent,thm:eigenvalues_and_functions}, we have $\{f^{(2m,2)} \circ \psi \colon m \in\Z \}$ forms an orthonormal basis for $L^{2}_{\Lambda + 2 \alpha \delta_{1}}$, and thus,  $\{f^{(2m,2)}\rvert_{(0,1/2]} \colon m \in\Z \}$ is an orthonormal basis for $L^{2}_{\Lambda \rvert_{(0,1/2]} +  \alpha \delta_{1/2}}$. This implies that $\langle h, f \rangle_{\eta} = 0$ for all $f \in \operatorname{span} \{ f^{(2m,2)} \colon m \in \Z \}$.  Using this, we obtain, for all Borel sets $A \subset (0, 1/2]$, that
\begin{align*}
\int \mathds{1}_{A}(x)  (h(x) + h(x+1/2) )  \, \mathrm{d}{\eta}(x) = 0.
\end{align*}
This yields $\eta(\{x \in (0,1/2] \colon h(x) > - h(x+1/2) \}) =\eta(\{x \in (0,1/2] \colon h(x) < -h(x+1/2) \})=0$. Hence, we have $h(x) = - h(x+1/2)$ for $\eta$-almost all $x \in (0, 1/2]$. Since $h \in \mathscr{D}_{\eta}^{2}$, it is left-continuous, which implies that $h(x) = - h(x+1/2)$ for all $x \in (0,1/2]$.  Therefore,
\begin{align*}
h(x) &=
\begin{cases}
\sin(b x + \gamma_{1}) & \text{if} \; x \in (0, 1/2],\\
-\sin(b (x-1/2) + \gamma_{1}) & \text{if} \; x \in (1/2, 1],
\end{cases}
\end{align*}
contradicting our initial assumption that $\lvert a_{2} \rvert < 1$.

For $k_{1},k_{2} \in \Z$ with $k_{1} \neq k_{2}$ we observe that $\gamma^{(k_{1},1)} \neq \gamma^{(k_{2},1)}$. Further, an elementary calculation shows that $(b^{(k_{1},1)})^{2} \neq (b^{(k_{2},1)})^{2}$, which implies all eigenvalues of $\Delta_{\eta}$ have multiplicity one.  Combining the above we obtain the following.

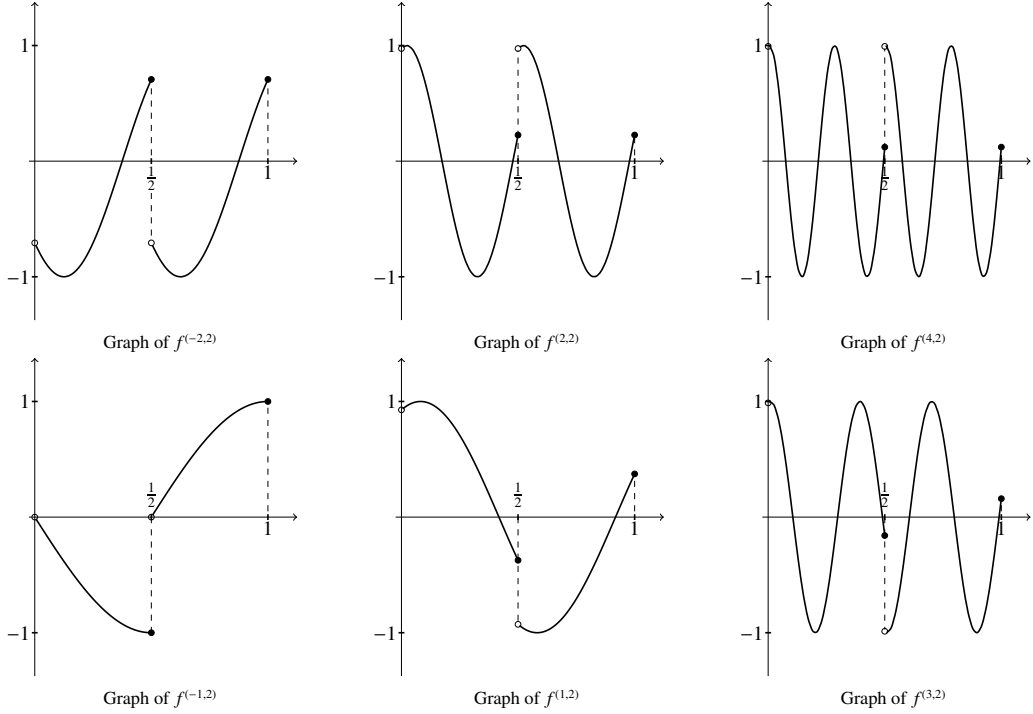
\begin{figure}[t]
\centering
\resizebox{0.875\linewidth}{!}{
\begin{subfigure}[t]{5.5cm}
\begin{tikzpicture}
%axis
  \draw[->] (-0.1,0) -- (4.5,0); 
  \draw[->] (0,-2.75) -- (0,2.75);
%atoms  
  \draw[thick,smooth] (2,-0.05) -- (2,0.05) node[below] {$\frac{1}{2}$};   
  \draw[thick,smooth] (4,-0.05) -- (4,0.05) node[below] {$1$};   
%function values  
   \draw[thick,smooth] (-0.05,2) -- (0.05,2) node[left] {$1$};  
   \draw[thick,smooth] (-0.05,-2) -- (0.05,-2) node[left] {$-1$}; 
%function  
  \draw[thick,domain=0.02:2,smooth, variable=\a] plot (\a,{2*sin ((-6.283)*(\a/4 r) - 0.785 r )});
  \draw[thick,domain=2.02:4,smooth, variable=\a] plot (\a,{2*sin ((-6.283)*(\a/4 r) + 8.639 r )});  
%circles  
  \filldraw (2,1.414) circle (1.4pt);
  \draw (0,-1.414) circle (1.4pt);
  \filldraw (4,1.414) circle (1.4pt);
  \draw (2,-1.414) circle (1.4pt);
%dashed lines 
  \draw[dashed,smooth] (4,0.1) -- (4,1.414);  
  \draw[dashed,smooth] (2,0.1) -- (2,1.414);  
  \draw[dashed,smooth] (2,-0.55) -- (2,-1.384);  
   \end{tikzpicture}
   \subcaption{Graph of $f^{(-2,2)}$} 
   \end{subfigure}
\qquad
\begin{subfigure}[t]{5.5cm}
\begin{tikzpicture}
%axis
  \draw[->] (-0.1,0) -- (4.5,0); 
  \draw[->] (0,-2.75) -- (0,2.75);
%atoms  
  \draw[thick,smooth] (2,-0.05) -- (2,0.05) node[below] {$\frac{1}{2}$};   
  \draw[thick,smooth] (4,-0.05) -- (4,0.05) node[below] {$1$};   
%function values  
   \draw[thick,smooth] (-0.05,2) -- (0.05,2) node[left] {$1$};  
   \draw[thick,smooth] (-0.05,-2) -- (0.05,-2) node[left] {$-1$}; 
%function  
  \draw[thick,domain=0.025:2,smooth, variable=\a] plot (\a,{2*sin ((10.340)*(\a/4 r) + 1.342 r)});
  \draw[thick,domain=2.027:4,smooth, variable=\a] plot (\a,{2*sin ((10.340)*(\a/4 r) - 10.112 r )});  
%circles  
  \filldraw (4,0.452) circle (1.4pt);
  \draw (0,1.948) circle (1.4pt);
  \filldraw (2,0.452) circle (1.4pt);
  \draw (2,1.948) circle (1.4pt);
%dashed lines 
  \draw[dashed,smooth] (2,0) -- (2,1.918);  
  \draw[dashed,smooth] (4,0.1) -- (4,0.452);   
   \end{tikzpicture}
   \subcaption{Graph of $f^{(2,2)}$}
   \end{subfigure}
\qquad
\begin{subfigure}[t]{5.5cm}
\begin{tikzpicture}
%axis
  \draw[->] (-0.1,0) -- (4.5,0); 
  \draw[->] (0,-2.75) -- (0,2.75);
%atoms  
  \draw[thick,smooth] (2,-0.05) -- (2,0.05) node[below] {$\frac{1}{2}$};   
  \draw[thick,smooth] (4,-0.05) -- (4,0.05) node[below] {$1$};   
%function values  
   \draw[thick,smooth] (-0.05,2) -- (0.05,2) node[left] {$1$};  
   \draw[thick,smooth] (-0.05,-2) -- (0.05,-2) node[left] {$-1$}; 
%function  
  \draw[thick,domain=0.02:2,smooth, variable=\a] plot (\a,{2*sin ((22.478)*(\a/4 r) + 1.449 r )});
  \draw[thick,domain=2.02:4,smooth, variable=\a] plot (\a,{2*sin ((22.478)*(\a/4 r) - 22.356 r )});  
%circles  
  \filldraw (4,0.243) circle (1.4pt);
  \draw (0,1.985) circle (1.4pt);
  \filldraw (2,0.243) circle (1.4pt);
  \draw (2,1.985) circle (1.4pt);
%dashed lines 
  \draw[dashed,smooth] (2,0) -- (2,1.945);  
  \draw[dashed,smooth] (4,0) -- (4,0.243);   
   \end{tikzpicture}
   \subcaption{Graph of $f^{(4,2)}$} 
   \end{subfigure}}
\centering
\resizebox{0.875\linewidth}{!}{
\begin{subfigure}[t]{5.5cm}
\begin{tikzpicture}
%axis
  \draw[->] (-0.1,0) -- (4.5,0); 
  \draw[->] (0,-2.75) -- (0,2.75);
%atoms  
  \draw[thick,smooth] (2,-0.05) -- (2,0.05) node[above] {$\frac{1}{2}$};   
  \draw[thick,smooth] (4,-0.05) -- (4,0.05) node[below] {$1$};   
%function values  
   \draw[thick,smooth] (-0.05,2) -- (0.05,2) node[left] {$1$};  
   \draw[thick,smooth] (-0.05,-2) -- (0.05,-2) node[left] {$-1$}; 
%function  
  \draw[thick,domain=0.02:2,smooth, variable=\a] plot (\a,{2*sin ((-3.142)*(\a/4 r))});
  \draw[thick,domain=2.02:4,smooth, variable=\a] plot (\a,{2*sin ((-3.142)*(\a/4 r) + 4.712 r )});  
%circles  
  \filldraw (4,2) circle (1.4pt);
  \draw (0,0) circle (1.4pt);
  \filldraw (2,-2) circle (1.4pt);
  \draw (2,0) circle (1.4pt);
%dashed lines 
  \draw[dashed,smooth] (2,0) -- (2,-2);  
  \draw[dashed,smooth] (4,0.1) -- (4,2);   
   \end{tikzpicture}
   \subcaption{Graph of $f^{(-1,2)}$} 
   \end{subfigure}
\qquad
\begin{subfigure}[t]{5.5cm}
\begin{tikzpicture}
%axis
  \draw[->] (-0.1,0) -- (4.5,0); 
  \draw[->] (0,-2.75) -- (0,2.75);
%atoms  
  \draw[thick,smooth] (2,-0.05) -- (2,0.05) node[above] {$\frac{1}{2}$};   
  \draw[thick,smooth] (4,-0.05) -- (4,0.05) node[below] {$1$};   
%function values  
   \draw[thick,smooth] (-0.05,2) -- (0.05,2) node[left] {$1$};  
   \draw[thick,smooth] (-0.05,-2) -- (0.05,-2) node[left] {$-1$}; 
%function  
  \draw[thick,domain=0.02:2,smooth, variable=\a] plot (\a,{2*sin ((4.673)*(\a/4 r) + 1.188 r)});
  \draw[thick,domain=2.025:4,smooth, variable=\a] plot (\a,{2*sin ((4.673)*(\a/4 r) - 4.290 r )});  
%circles  
  \filldraw (4,0.747) circle (1.4pt);
  \draw (0,1.855) circle (1.4pt);
  \filldraw (2,-0.747) circle (1.4pt);
  \draw (2,-1.855) circle (1.4pt);
%dashed lines 
  \draw[dashed,smooth] (2,0) -- (2,-1.815);  
  \draw[dashed,smooth] (4,0.1) -- (4,0.747);   
   \end{tikzpicture}
   \subcaption{Graph of $f^{(1,2)}$} 
   \end{subfigure}
\qquad
\begin{subfigure}[t]{5.5cm}
\begin{tikzpicture}
%axis
  \draw[->] (-0.1,0) -- (4.5,0); 
  \draw[->] (0,-2.75) -- (0,2.75);
%atoms  
  \draw[thick,smooth] (2,-0.05) -- (2,0.05) node[above] {$\frac{1}{2}$};   
  \draw[thick,smooth] (4,-0.05) -- (4,0.05) node[below] {$1$};   
%function values  
   \draw[thick,smooth] (-0.05,2) -- (0.05,2) node[left] {$1$};  
   \draw[thick,smooth] (-0.05,-2) -- (0.05,-2) node[left] {$-1$}; 
%function  
  \draw[thick,domain=0.02:2,smooth, variable=\a] plot (\a,{2*sin ((16.347)*(\a/4 r) + 1.411 r)});
  \draw[thick,domain=2.03:4,smooth, variable=\a] plot (\a,{2*sin ((16.347)*(\a/4 r) - 16.187 r )});  
%circles  
  \filldraw (4,0.319) circle (1.4pt);
  \draw (0,1.975) circle (1.4pt);
  \filldraw (2,-0.319) circle (1.4pt);
  \draw (2,-1.975) circle (1.4pt);
%dashed lines 
  \draw[dashed,smooth] (2,0) -- (2,-1.945);  
  \draw[dashed,smooth] (4,0.1) -- (4,0.319);   
   \end{tikzpicture}
   \subcaption{Graph of $f^{(3,2)}$}
   \end{subfigure}}
\caption{\parbox[t]{30em}{Graphs of the eigenfunctions $f^{(k,2)}$ of $\Delta_{\eta}$, for $k \in \{ -2, -1,1, 2, 3,4\}$, where $\eta = \Lambda  + \pi^{-1} \delta_{1/2} + \pi^{-1} \delta_{1}$.}}
\label{fig:ef_2_atom}
\end{figure}

\begin{thm}\label{thm:two_atom_case_thm}
The eigenvalues of $\Delta_{\eta}$ are $\lambda^{(k,2)} \coloneqq -(b^{(k,2)})^2$ for $k\in\Z$, with corresponding eigenfunctions $f^{(k,2)}$.  Further, each eigenvalue has multiplicity one.
\end{thm}

Notice the only eigenfunction $f^{(k,2)}$ with $\mathbf{f}^{(k,2)}$ continuous at an atom is the constant function $f^{(0,2)}$. Indeed, if there exists a $k \in \Z \setminus \{ 0 \}$ such that $\mathbf{f}^{(k,2)}$ is continuous at an atom, we would have $\cos(\gamma_{1}^{(k,2)}) = 0$ or $\cos(b^{(k,2)}/2+\gamma_{2}^{(k,2)}) = 0$. Substituting the defining equations for $b^{(k,2)}$ and $\gamma_{2}^{(k,2)}$ into the latter, we obtain, in both cases, that $\gamma_{1}^{(k,2)} = \pi/2$, which contradicts the fact that $\gamma_{1}^{(k,2)} \in (-\pi/2, \pi/2)$.  We also note the following.
\begin{enumerate}
\item $\displaystyle \lim_{k \to +\infty} \gamma_{1}^{(k,2)} = - \lim_{k \to -\infty} \gamma_{1}^{(k,2)} = \pi/2$\\
\item $\displaystyle \lim_{k \to +\infty} \gamma_{2}^{(k,2)} \bmod 2\pi = \pi $ and $\displaystyle \lim_{k \to -\infty} \gamma_{2}^{(k,2)} \bmod 2 \pi  = 0$\\
\item $\displaystyle \lim_{k \to +\infty} b^{(k,2)}/(2\pi k) = \lim_{k \to -\infty} b^{(k,2)}/(2\pi k) = 1$
\end{enumerate}

\begin{cor}\label{cor:evcf_2atom}
Letting $N_{\eta} \colon \R^{+} \to \R$ denote the eigenvalue counting function of $-\Delta_{\eta}$, we have $\displaystyle \lim_{x\to \infty} \frac{\pi N_{\eta}(x)}{\sqrt{x}} = 1$.
\end{cor}

 In contrast to the case when $\eta$ is atomless and the case when $N=1$, the eigenvalues of $\Delta_{\eta}$ do not occur in pairs. However, we have that $\lim_{k \to \infty} \, -b^{(k,2)}/b^{(-k-1,2)} = 1$.

\begin{ex}
For $\eta = \Lambda + \pi^{-1}\delta_{\frac{1}{2}} + \pi^{-1}\delta_{1}$, we have that $\lambda^{(-2,2)} = 4 \pi^2$, $\lambda^{(-1,2)} = \pi^2$, $\lambda^{(1,2)} \approx 21.8$, $\lambda^{(2,2)} \approx 106.9$, $\lambda^{(3,2)} \approx 267.2$ and $\lambda^{(4,2)} \approx 505.3$; see \Cref{fig:ef_2_atom} for a graphical representation of $f^{(k,2)}$ for $k \in \{ -2, -1, 1, 2, 3, 4\}$.
\end{ex}

\section*{Acknowledgements}
	
The authors thank Hendrik Vogt for several enlightening discussions. They acknowledge the support of the Deutsche Forschungsgemeinschaft (DFG grant Ke\,1440/3-1). Part of this work was completed while the first and second authors were visiting Institut Mittag-Leffler as part of the research program \textsl{Fractal Geometry and Dynamics}. They are extremely grateful to the organisers and staff for their very kind hospitality, financial support and a stimulating atmosphere. The third author acknowledges the support from the BremenIDEA program at Universit{\"a}t Bremen. 

\bibliographystyle{plain}
\bibliography{bib}

\end{document}